\documentclass[a4paper,11pt]{amsart}
\pdfoutput=1

\usepackage{dsfont}
\usepackage{multicol}	     
\usepackage[bottom]{footmisc}

\usepackage[title]{appendix}

\usepackage{pgf, tikz, tikz-cd}
\usepackage{dsfont}
\usetikzlibrary{matrix}
\usetikzlibrary{braids}
\usetikzlibrary{arrows}
\usetikzlibrary{patterns}
\usepackage{braids}

\usepackage{mathtools}
\mathtoolsset{showonlyrefs=true}

\usepackage{amsmath,amssymb,amsthm,amsfonts, mathrsfs, fancyhdr}

\usepackage{anyfontsize}

\usepackage{upgreek}

\usepackage{enumitem}
\newcommand{\mylabel}[2]{#2\def\@currentlabel{#2}\label{#1}}
\makeatother

\usepackage{todonotes}

\setlist[description]{leftmargin=*}
\usepackage{amscd}
\usepackage[active]{srcltx}
\usepackage{verbatim}
\usepackage{epsfig,graphicx,color}
\usepackage{graphicx}
\usepackage{mathtools,xparse}
\usepackage[english]{babel}

\usepackage{tabularx,ragged2e}
\newcolumntype{L}{>{\RaggedRight\arraybackslash}X}

\usepackage{xltabular}
\usepackage{booktabs}
\usepackage{caption}

\usepackage{pgf, tikz, tikz-cd}
\usepackage{dsfont}
\usetikzlibrary{matrix}

\usepackage{capt-of}

\usepackage{subcaption}

\usepackage{tkz-euclide}
\usetikzlibrary{arrows}
\usetikzlibrary{calc}
\usetikzlibrary{patterns}
\usepackage{tikz-cd}
\usepackage{mathtools}
\usepackage{enumitem}
\usepackage{float}
\usepackage{empheq}

\usepackage{fouriernc}

\usepackage[left=2.7cm,right=2.7cm,top=3cm,bottom=3cm]{geometry}

\usepackage[unicode,bookmarks, pdftex, backref = page]{hyperref}
\definecolor{newblue}{RGB}{0,102,204}
\definecolor{newred}{RGB}{206,32,41}

\definecolor{qqqqcc}{rgb}{0,0,0.8}
\definecolor{ffqqqq}{rgb}{1,0,0}
\definecolor{rvwvcq}{rgb}{0.08235294117647059,0.396078431372549,0.7529411764705882}

\hypersetup{colorlinks=true,citecolor=newblue,linkcolor=newred,
  urlcolor=magenta,
pdfpagemode=UseNone, breaklinks=true}

\usepackage{apptools}
\AtAppendix{\counterwithin{theorem}{section}}

\newtheorem{theorem}{Theorem}[section]
\newtheorem{proposition}[theorem]{Proposition}
\newtheorem{lemma}[theorem]{Lemma}
\newtheorem{corollary}[theorem]{Corollary}

\newtheorem{letterthm}{Theorem}

\theoremstyle{definition}
\newtheorem{remark}[theorem]{Remark}
\newtheorem{example}[theorem]{Example}
\newtheorem{definition}[theorem]{Definition}

\renewcommand{\setminus}{-}

\renewcommand{\to}{\longrightarrow}

\newcommand{\B}{\mathsf{B}}
\newcommand{\C}{\mathsf{C}}
\newcommand{\D}{\mathsf{D}}
\renewcommand{\a}{\mathsf{a}}
\renewcommand{\b}{\mathsf{b}}
\newcommand{\s}{\mathsf{s}}

\DeclareMathOperator{\Sym}{Sym}
\DeclareMathOperator{\Aut}{Aut}

\usepackage{epigraph}
\usepackage{listings}
\usepackage{color}

\lstdefinestyle{customc}{
  belowcaptionskip=1\baselineskip,
  breaklines=true,
 frame=L,
  xleftmargin=\parindent,
  language=GAP,
  showstringspaces=false,
  basicstyle=\tiny\ttfamily,
  keywordstyle=\bfseries\color{green!40!black},
  commentstyle=\itshape\color{purple!40!black},
  identifierstyle=\color{blue},
  stringstyle=\color{olive},
}

\definecolor{codegreen}{rgb}{0,0.6,0}
\definecolor{codegray}{rgb}{0.5,0.5,0.5}
\definecolor{codepurple}{rgb}{0.58,0,0.82}
\definecolor{backcolour}{rgb}{0.95,0.95,0.92}

\lstdefinestyle{mystyle}{
    backgroundcolor=\color{backcolour},   
    commentstyle=\color{codegreen},
    keywordstyle=\color{magenta},
    numberstyle=\tiny\color{codegray},
    stringstyle=\color{codepurple},
    basicstyle=\ttfamily\tiny,
    breakatwhitespace=false,         
    breaklines=true,                 
    captionpos=t,                    
    keepspaces=true,                 
    numbers=left,                    
    numbersep=5pt,                  
    showspaces=false,                
    showstringspaces=false,
    showtabs=false,                  
    tabsize=2
}

\title[Finite quotients of full surface braid groups]{Finite quotients of full surface braid groups\\and complex surfaces of general type:\\cyclic, dihedral, and extra-special quotients}

\date{\today}

\author[Massimiliano Alessandro]{Massimiliano Alessandro}
\address{Fachrichtung Mathematik 
\newline \indent
Campus, Geb\"{a}ude E2 4 
\newline \indent
Universit\"{a}t des Saarlandes 
\newline \indent
66123 Saarbr\"{u}cken, Germany} 
\email{malessandro@math.uni-sb.de}

\author[Michelangelo Migliano]{Michelangelo Migliano}
\email{michelangelo.migliano@gmail.com}

\author[Francesco Polizzi]{Francesco Polizzi $^{*}$}
\address{Dipartimento di Matematica e Applicazioni ``Renato Caccioppoli''
  \newline\indent
  Universit\`a degli Studi di Napoli Federico II
  \newline\indent
Via Cintia, Monte S. Angelo
  \newline\indent
  I-80126  Napoli, Italy}
\email{francesco.polizzi@unina.it}



\begin{document}

\begin{abstract}
Let $\mathsf{B}_2(\Sigma_g)$ be the full braid group on two strings on a compact Riemann surface of genus $g$. We compute the number of epimorphisms $\varphi \colon \mathsf{B}_2(\Sigma_g) \to G$, where $G$ is a finite cyclic, dihedral or extra-special group, under the assumption that $\varphi$ does not factor through $\pi_1(\operatorname{Sym}^2\Sigma_g)$. We then apply our algebraic results to the geometric problem of constructing smooth surfaces of general type as Galois covers of $\operatorname{Sym}^2\Sigma_g$ branched along the diagonal. Notably, when the branching order is at least $3$, we obtain surfaces with positive signature, i.e. $c_1^2>2c_2$. In particular, we construct two $3$-dimensional families of minimal surfaces of general type with $p_g(S)=154$, $q(S)=2$ and $K_S^2=1269$, such that members of different families have the same numerical invariants, the same Betti numbers and the same positive signature, but different torsion subgroups in their first integral homology groups.
\end{abstract}

\maketitle

\section*{Introduction} \label{sec:intro}

Let $\Sigma_g$ be a compact Riemann surface of genus $g$ and let $\mathsf{B}_2(\Sigma_g)$ denote the full braid group on two strings of $\Sigma_g$, which can be interpreted as the fundamental group $\pi_1(\Sym^2 \Sigma_g - \delta)$, where $\delta \subset \Sym^2(\Sigma_g)$ is the diagonal.
Full surface braid groups play an important role in low-dimensional topology and related areas, connecting ideas from algebraic topology, group theory, and algebraic geometry.
These groups are finitely presented for every choice of $g$, and several presentations appear in the literature, see for instance \cite{Scott70, GM01, Bel04}; in the present work we use the presentation with generators $a_1, \, b_1, \dots, a_g, \, b_g, \sigma$ given in \cite[Theorem 1.2]{Bel04}, see Theorem \ref{thm:Bellingeri}.

In the sequel, we always assume $g\geq 2$.
Given a finite group $G$, a group epimorphism $\varphi\colon \mathsf{B}_2(\Sigma_g)\to G$ is called an \emph{admissible epimorphism} of type $(g, \, n)$ if $\varphi(\sigma)$ has order $n\geq 2$.
By using the isomorphism $\mathsf{B}_2(\Sigma_g)\simeq \pi_1(\Sym^2 \Sigma_g - \delta)$ one sees that $\varphi\colon \mathsf{B}_2(\Sigma_g)\to G$ is admissible if and only if $\varphi$ does not factor through $\pi_1(\Sym^2 \Sigma_g)$, see Remark \ref{rmk:admissible_epimorphism}.
Let $N_g(G)$ denote the total number of admissible epimorphisms $\varphi\colon \mathsf{B}_2(\Sigma_g)\to G$.
In this paper we compute the number $N_g(G)$ when $G$ is cyclic, dihedral and extra-special. Our algebraic investigation was also motivated by geometric considerations that we explain below.

In terms of the isomorphism $\mathsf{B}_2(\Sigma_g) \simeq \pi_1(\Sym^2 \Sigma_g - \delta)$, the generator $\sigma$ corresponds to the homotopy class in $\Sym^2 \Sigma_g - \delta$ of a topological loop that ``winds once around the diagonal $\delta$''. By the Grauert-Remmert extension theorem (see \cite{GrRem58}, \cite[XII.5.4]{SGA1}, \cite{DeGr94}) and Serre's GAGA (see \cite{Serre56}, \cite[Chapter 6]{Serre08}) the set of admissible epimorphisms $\varphi\colon \mathsf{B}_2(\Sigma_g)\to G$ of type $(g, \, n)$, up to automorphisms of $G$, is in bijection with the set of isomorphism classes of Galois covers
\begin{equation}\label{G-cover}
f \colon S \to \operatorname{Sym}^2 \Sigma_g,
\end{equation}
with Galois group $G$, branched precisely along $\delta$ with branching order $n$, cf. \cite[Section 1]{Pol18}.
Note that in \eqref{G-cover} the total space $S$ is a smooth projective surface, since $\operatorname{Sym}^2 \Sigma_g$ is so and $\delta \subset \operatorname{Sym}^2 \Sigma_g$ is a smooth divisor. Furthermore, since $g\geq 2$ one sees that $S$ is a minimal surface of general type whose invariants can be computed in terms of $|G|$, $g$ and $n$, see Proposition \ref{prop:invariants_S}. It is worth mentioning here that when $n=2$ we have surfaces with zero signature, namely with $c_1^2(S)=2c_2(S)$. On the other hand, when $n\geq 3$, we get surfaces with positive signature, i.e. lying in the \emph{arctic zone} $c_1^2(S)>2c_2(S)$, see Proposition \ref{prop:slope_S}.
\vspace{0.3cm}

\subsection*{Main results}
We can now state our main results; recall that we are assuming $g \geq 2$ everywhere.

\begin{letterthm}[{Theorem \ref{thm:cyclic}}]
Let $\C_m$ be the cyclic group of order $m$. Then the number
$N_g(\C_m)$ of admissible epimorphisms
$\varphi \colon \B_2(\Sigma_g) \to \C_m$ is given by
\begin{equation}
N_g(\C_m)=
\begin{cases}
0 & \text{if } 2\nmid m,\\[2mm]
2^{2g}J_{2g}\left(\dfrac{m}{2}\right)
& \text{if } 2\mid m.
\end{cases}
\end{equation}
\end{letterthm}

\begin{letterthm} [{Theorem \ref{thm:dihedral}}]
Let $m\geq 2$ be an integer and let $\D_{2m}$ be the dihedral group of order $2m$. Then the number $N_g(\D_{2m})$ of admissible epimorphisms $\varphi \colon \mathsf{B}_2(\Sigma_g) \to \D_{2m}$ is given by 
\begin{equation} 
N_g(\D_{2m}) = \varepsilon \cdot mJ_{2g}(m), \quad 
\mathrm{where } \; \;  \varepsilon =
\begin{cases}
1 &\text{if } \, m \notin \{2, \, 4\} \\
2 &\text{if } \, m=4 \\
3 \cdot 2^{2g-1} & \text{if } \, m=2.
\end{cases}
\end{equation}
\end{letterthm}

Furthermore, for any admissible epimorphism $\varphi\colon \mathsf{B}_2(\Sigma_g) \to G$, with $G$ either abelian or dihedral, the order of $\varphi(\sigma)$ is always $2$.

By contrast, in Section \ref{sec:_extraspecial_quotients} we 
show that for every extra-special admissible quotient $\varphi \colon \B_2(\Sigma_g) \to G$,  with $|G|=p^{2g+1}$, the image $\varphi(\sigma)$ is never an involution of $G$. More precisely, if $p$ is odd then $\varphi(\sigma)$ belongs to $Z(G)$, and so has order $p$. Instead, if $p=2$ then (rather surprisingly) the element $\varphi(\sigma)$ has order $4$, and  $\varphi(\sigma)^2$ is the unique central involution of $G$.  This is a striking difference from  the abelian and dihedral cases. 

In order to state our counting results for extra-special quotients, let us recall that, for any  integer $g\geq 2$ and any prime number $p$, there are exactly two isomorphism classes of  extra-special $p$-groups of order $p^{2g+1}$.  They are represented by the groups  $\mathsf{H}_{2g+1}(\mathbb{F}_p)$ and $\mathsf{G}_{2g+1}(\mathbb{F}_p)$, whose definition can be found in Proposition \ref{prop:extra-special-groups}.

\begin{letterthm} [Proposition \ref{prop:extra-special_congruence} and Theorem \ref{thm:extraspecial}] \label{thm:C}
 Let $G$ be an extra-special group of order $p^{2g+1}$, with $p$ \emph{odd}. Then $G$ is an admissible quotient of $\B_2(\Sigma_g)$ if and only if $g \equiv -1 \pmod{p}$.  In this case, the number $N_g(G)$ of admissible epimorphisms $\varphi \colon \B_2(\Sigma_g) \to G$ is given by
  \begin{equation} 
 N_g(G) =  (p-1) \, p^{g(g+2)} \prod_{i=1}^g (p^{2i} -1).
 \end{equation}
In particular, $N_g(G)$ is independent of $G$.
\end{letterthm}

\begin{letterthm}[Propositions \ref{prop:extra-special_2_congruence} and \ref{prop:A(G)} and Theorem \ref{thm:2-extraspecial}] 
\label{thm:D}
Let $G$ be an extra-special group of order $2^{2g+1}$. Then $G$ is an admissible quotient of $\B_2(\Sigma_g)$ if and only if $g$ is odd.  In this case, the number $N_g(G)$ of admissible epimorphisms $\varphi \colon \B_2(\Sigma_g) \to G$ is given by
\begin{equation} 
N_g(G) =  2^{2g+1} A(G) \cdot |\mathsf{Sp}(2g, \, \mathbb{F}_2)|,
\end{equation}
where $\mathsf{Sp}(2g, \, \mathbb{F}_2)$ denotes the symplectic group of $2g \times 2g$ matrices over the field $\mathbb{F}_2$ and 
\begin{equation} 
A(G)=
  \begin{cases}
  2^{2g-1} -  2^{g-1}   & \textrm{if }
    G=\mathsf{H}_{2g+1}(\mathbb{F}_2) \\
   2^{2g-1} +  2^{g-1}&
    \textrm{if }
    G=\mathsf{G}_{2g+1}(\mathbb{F}_2).
  \end{cases}
\end{equation}
\end{letterthm}

The proofs of  Theorems C and D rely heavily on the natural symplectic structure of the elementary abelian quotient $G/Z(G) \simeq \mathbb{F}_p^{2g}$. Additionally, when $p=2$, an important role is played by the Arf invariant of the quadratic form $q \colon G/Z(G) \to \mathbb{F}_2$.

In each case, we also provide the number $N_g^{\circ}(G) = N_g(G)/|\Aut(G)|$ of admissible epimorphisms up to automorphisms of $G$, see Remark \ref{rmk:action_auto} and Corollaries \ref{cor:cyclic_up_to_auto}, \ref{cor:dihedral_up_to_auto}, \ref{cor:extraspecial_auto}, \ref{cor:2-extraspecial_auto}. An interesting dichotomy emerges in the case of extra-special admissible quotients. For odd $p$, the quantity $N_g(G)$ is independent of $G$, whereas $N_g^\circ(G)$ depends on it, and so distinguishes the Heisenberg and non-Heisenberg types. For $p=2$, the opposite phenomenon occurs, namely $N_g(G)$ depends on the Arf invariant of $q$, and so on $G$, while $N_g^\circ(G)$ is independent of $G$.
\medskip

Finally, in Section \ref{sec:H_1} we explain how to compute the first homology group $H_1(S, \, \mathbb{Z})$ of the $G$-covers $S$ as in \eqref{G-cover}. As an application, we construct two $3$-dimensional families of minimal surfaces of general type such that members of different families have the same numerical invariants
\begin{equation}
p_g=154,  \quad q=2, \quad K^2=1269,
\end{equation}
 the same Betti numbers, the same positive signature but different torsion subgroups  of  $H_1(S, \, \mathbb{Z})$, see Proposition \ref{prop:invariants_H_G} and Remark \ref{rmk:invariants_H_G}.

The calculation of $H_1(S, \, \mathbb{Z})$ was performed using the computer algebra system \verb|GAP4| \cite{GAP4}. All relevant scripts are available in a public GitLab repository \cite{GitLab26}, and selected excerpts of the code are included in the paper to illustrate the computational methods involved.

\subsection*{Related work}
Finite quotients of surface braid groups have recently been investigated by Tan \cite{Tan24}, who determined the smallest nonabelian quotients of the full surface braid group on $k$ strings $\B_k(\Sigma_g)$ for $k\geq 5$.\footnote{When $k=2,\,3,\,4$, Tan remarks that the smallest nonabelian quotient of $\B_k(\Sigma_g)$ is $S_3$.}

An important role in Tan's work is played by braid-free quotients, namely those for which the image of the classical braid subgroup is cyclic. When $k\geq 3$, every such quotient is two-step nilpotent, and the image of the standard braid generators is central.

The case $k=2$, considered in the present paper, exhibits a substantially different behavior. Indeed, the braid-free condition becomes automatic, and the image of the unique standard braid generator need not be central. In particular, we obtain non-nilpotent dihedral quotients (for instance, $D_6\simeq S_3$), as well as extra-special quotients of order $2^{2g+1}$ with a noncentral image of the standard braid generator.

There is nevertheless a close connection in odd characteristic. The two extra-special groups $\mathsf{H}_{2g+1}(\mathbb{F}_p)$ and $\mathsf{G}_{2g+1}(\mathbb{F}_p)$ of order $p^{2g+1}$ appearing in our results coincide with the groups $\mathrm{I}(p,g)$ and $\mathrm{II}(p,g)$ considered in \cite[Section 2]{Tan24}. Moreover, our condition $p\mid(g+1)$ agrees with the formal specialization to $k=2$ of the congruence $p\mid(g+k-1)$ arising in Tan's braid-free construction when $k\geq 3$. Our results provide, in addition, exact formulas for the number of admissible epimorphisms in each case.

A more pronounced difference occurs in characteristic $2$. Whereas the $2$-groups arising in Tan's minimal order construction have center of order at least $4$, our admissible extra-special quotients of order $2^{2g+1}$ have center of order $2$, and the image of the standard braid generator is a noncentral element of order $4$. The corresponding enumeration is governed by the Arf invariant of a quadratic form over $\mathbb{F}_2$.

\medskip

Finally, in a similar spirit to the present work, a series of recent works \cite{CaPol21, Pol22, PolSab22, PolSab23, PolSab25, PolSab26} has shed light on the role of admissible quotients of the pure braid group $\mathsf{P}_2(\Sigma_g)$ in algebraic geometry, and in particular in the theory of double Kodaira surfaces (see \cite{Rol10, Cat17, LLR20} for general background on this last topic). A comparison between the results obtained therein and those presented in the present article would be of independent interest, but lies beyond the scope of this work.

\subsection*{Notation and conventions}

If $S$ is a smooth projective surface over $\mathbb{C}$, then $c_1(S)$, $c_2(S)$ denote the first and second Chern classes of its tangent bundle $T_S$, respectively. Moreover, $K_S$ is the canonical class, $p_g(S)=h^0(S, \, K_S)$ is the geometric genus, $q(S)=h^1(S, \, \mathcal{O}_S)$ is the irregularity and $\chi(\mathcal{O}_S)$ is the holomorphic Euler characteristic.

If $X$ is a topological space, we denote by $\chi_{\textrm{top}}(X)$ its topological Euler characteristic.

Throughout the paper we use the following notation for
groups:
\begin{itemize}
  \item $\C_m$ is the cyclic group of order $m$, $\D_{2m}$ is the dihedral group of order $2m$ and $\mathsf{Q}_8$ is the quaternion group of order $8$.
  
  \item $\mathsf{Sp}(2n, \, \mathbb{F}_q)$ is the symplectic group of $2n \times 2n$ matrices over a field with $q$ elements. 
  
  \item The order of a finite group $G$ is denoted by $|G|$. If $x \in G$, the order of $x$ is denoted by $o(x)$ and its centralizer in $G$ by $C_G(x)$.
  
  \item $\operatorname{Aut}(G)$ is the automorphism group of $G$.
  
  \item If $x, \, y \in G$, their commutator is defined as $[x,\, y]=xyx^{-1}y^{-1}$.
  
  \item The commutator subgroup of $G$ is denoted by $[G, \, G]$, the center of $G$ by $Z(G)$.
  
  \item If $S= \{s_1, \ldots, s_n \} \subset G$, the subgroup generated by $S$ is denoted by $\langle S \rangle=\langle s_1,\ldots, s_n \rangle$.
  
  \item If $N$ is a normal subgroup of $G$ and $g \in G$, we denote by $\bar{g}$ the image of $g$ in the quotient group $G/N$.
\end{itemize}

\section{Number-theoretic preliminaries}

Let $\mu \colon \mathbb{N}\rightarrow\mathbb{Z}$ be the \emph{M\"obius function} defined as 
\begin{equation} \label{eq:mobius_function}
\mu(n)=
\begin{cases}
1 & \text{if } n=1,\\ 
(-1)^l & \text{if }n \text{ is the product of }l\text{ distinct primes},\\ 
0 & \text{if } n \text{ is divisible by the square of a prime}.
\end{cases}
\end{equation}
We denote by $\C_m$ the cyclic group of order $m$, presented as $\C_m = \langle x \; | \; x^m=1 \rangle$.

\begin{proposition} \label{prop:Jordan_totient_1}
For any  $k \in \mathbb{N}$, the number $J_k(m)$ of ordered $k$-tuples $(x_1,\dots, x _k) \in (\C_m)^k$ generating $\C_m$ is given by 
\begin{equation} \label{eq:Jordan_totient_1}
J_k(m)=\sum_{d\mid m}d^k\mu\left(\frac{m}{d}\right).
\end{equation}
\end{proposition}
\begin{proof}
The group $\C_m$ contains precisely one subgroup isomorphic to $\C_d$ for every integer $d$ dividing $m$. So we can partition the $m^k$ ordered $k$-tuples of elements of $\C_m$ according to the subgroup that they generate. By definition, the number of ordered $k$-tuples that generate $\C_d$ is $J_k(d)$, so we get
\begin{equation} \label{eq:partitioning_m^k}
m^k = \sum_{d\mid m} J_k(d).
\end{equation}
Now, \eqref{eq:Jordan_totient_1} is obtained from \eqref{eq:partitioning_m^k} by applying M\"{o}bius inversion, see \cite[Theorem 6.14]{Nat00}.
\end{proof}

\begin{remark} \label{rmk:Jordan_totient}
The quantity $J_k(m)$  is called the \emph{Jordan totient}, and  can also be written as 
\begin{equation} \label{eq:Jordan_totient_2}
J_k(m)=m^k\prod_{p\mid m}\left(1-\frac{1}{p^k}\right),
\end{equation}
where the product ranges over the \emph{distinct} primes $p$ dividing $m$; see \cite[Exercise 1.5.2]{Mur01}. Comparing \eqref{eq:Jordan_totient_2} with the formula
\begin{equation} \label{eq:Euler_totient}
\phi(m)=m \prod_{p\mid m}\left(1-\frac{1}{p}\right),
\end{equation}
expressing the classical Euler totient \cite[Theorem 2.7]{Nat00}, we observe that $\phi(m)=J_1(m)$.
\end{remark}

\begin{remark}
Since $\operatorname{Aut}(\C_m)$ has order $\phi(m)$ and acts freely on generating ordered $k$-tuples, the number of such $k$-tuples, up to automorphisms, is
\begin{equation} \label{eq:k-tuples up to auto}
\frac{1}{\phi(m)} J_k(m)= m^{k-1} \prod_{p\mid m} \frac{1-p^{-k}}{1-p^{-1}}.
\end{equation}
\end{remark}

\section{Braid groups on compact Riemann surfaces}\label{sec:braids}

\begin{definition}
Let $\Sigma_g$ be a compact Riemann surface of genus $g$, and let  $\mathscr{P} = \{p_1, \,p_2\} \subset \Sigma_g$ be a pair of distinct points. A \emph{ braid} on $\Sigma_g$ based at $\mathscr{P}$ (also called a \emph{ braid on} $2$ \emph{strings}) is a pair $(\gamma_1, \, \gamma_2)$ of paths $\gamma_i \colon [0, \, 1] \to \Sigma_g$ such that 
\begin{itemize}
\item $\gamma_i(0) = p_i, \quad i=1, \, 2$;
\item $\gamma_i(1) \in \mathscr{P}, \quad i=1, \, 2$;
\item the points $\gamma_1(t), \, \gamma_2(t) \in \Sigma_g$ are distinct for all $t \in [0, \, 1]$.
\end{itemize}     
A braid such that $\gamma_i(0)=\gamma_i(1)$ for all $i \in \{1, \, 2\}$ is called a \emph{pure braid}, see Figure \ref{fig:braids}. 
\end{definition}
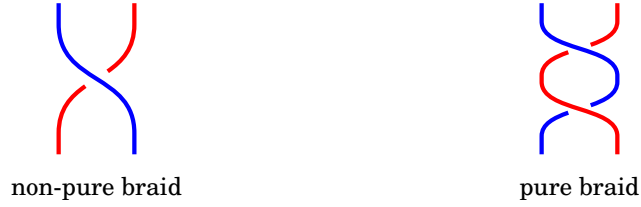
\begin{figure}[htbp] 
\centering

\begin{subfigure}[t]{0.30\textwidth}
\centering
\begin{tikzpicture}
\braid[
  number of strands=2,
  height=1.6cm,
  width=1cm,
  border height=0.2cm,
  style strands={1}{blue},
  style strands={2}{red},
  line width=1.6pt
] (braidleft) a_1;
\end{tikzpicture}
\caption*{Non-pure braid}
\end{subfigure}
\hspace{1.5cm}
\begin{subfigure}[t]{0.30\textwidth}
\centering
\begin{tikzpicture}
\braid[
  number of strands=2,
  height=0.8cm,
  width=1cm,
  border height=0.2cm,
  style strands={1}{blue},
  style strands={2}{red},
  line width=1.6pt
] (braidright) a_1 a_1;
\end{tikzpicture}
\caption*{Pure braid}
\end{subfigure}

\caption{Examples of a non-pure braid and a pure braid on two strings.}
\label{fig:braids}
\end{figure}

The \emph{full braid group}  on two strings on $\Sigma_g$ is the group $\mathsf{B}_{2}(\Sigma_g)$ whose elements are the homotopy classes of braids  based at $\mathscr{P}$ and whose operation is the usual product of paths, up to homotopies among braids. The \emph{pure braid group} is the subgroup $\mathsf{P}_{2}(\Sigma_g)$ of $\mathsf{B}_{2}(\Sigma_g)$ given by the homotopy classes of pure braids. It can be shown that $\mathsf{B}_{2}(\Sigma_g)$ and $\mathsf{P}_{2}(\Sigma_g)$ do not depend on the choice of the set $\mathscr{P}$, and that there is a short exact sequence of groups
\begin{equation} \label{eq:pure-nonpure}
1 \to \mathsf{P}_{2}(\Sigma_g)\to \mathsf{B}_{2}(\Sigma_g)\to \C_2 \to 1.
\end{equation}
Furthermore, there are isomorphisms
\begin{equation} \label{eq:iso-braids}
	\mathsf{P}_{2}(\Sigma_g) \simeq \pi_1(\Sigma_g \times \Sigma_g
		- \Delta,
	\, (p_1, \, p_2)), \quad
	\mathsf{B}_{2}(\Sigma_g) \simeq \pi_1(\Sym^2 \Sigma_g
		- \delta,
	\, \mathscr{P}),
\end{equation}
where $\Delta \subset \Sigma_g \times \Sigma_g$ and $\delta \subset \Sym^2 \Sigma_g$ are the diagonals. 
Thus, we can interpret \eqref{eq:pure-nonpure} as the short exact sequence of fundamental groups induced by the \'{e}tale double cover
\begin{equation}
\Sigma_g \times \Sigma_g - \Delta \to \operatorname{Sym}^2 \Sigma_g - \delta
\end{equation}
associated with the involution that exchanges the two coordinates in $\Sigma_g \times \Sigma_g - \Delta $.

In the sequel, we will only deal with $\mathsf{B}_2(\Sigma_g)$, always assuming $g \geq 2$. It is known that $\mathsf{B}_2(\Sigma_g)$ is finitely presented for all $g$, see for instance \cite{Scott70, GM01}; here we use the following presentation, which can be found in \cite[Theorem 1.2]{Bel04}.

\begin{theorem} \label{thm:Bellingeri}
The full braid group $\B_2(\Sigma_g)$ is generated by $2g+1$ elements 
\begin{equation}
a_1,b_1,\dots,a_g,b_g,\sigma
\end{equation}
 subject to the relations 
\begin{equation} 
\begin{aligned}[c]
[a_i, \; \sigma^{-1}a_i\sigma^{-1}]&=1\\
[b_i, \; \sigma^{-1}b_i\sigma^{-1}]&=1
\end{aligned}
\qquad\qquad
\begin{aligned}[c]
1\leq i\leq g
\end{aligned}
\end{equation}
\begin{equation}
\begin{aligned}[c]
[a_i, \; \sigma^{-1}a_j\sigma]&=1\qquad\\
[a_i, \; \sigma^{-1}b_j\sigma]&=1\\
[b_i, \; \sigma^{-1}a_j\sigma]&=1\\
[b_i, \; \sigma^{-1}b_j\sigma]&=1
\end{aligned}
\qquad\quad
\begin{aligned}[c]
j<i
\end{aligned}
\end{equation}
\begin{equation}
\begin{aligned}[c]
[a_r, \; \sigma^{-1}b_r \sigma^{-1}]&=\sigma^2
\end{aligned}
\qquad\;\;\:
\begin{aligned}[c]
1\leq r\leq g
\end{aligned}
\end{equation}
\begin{equation}
\begin{aligned}[c]
[a_1, \; b_1^{-1}]\cdots [a_g, \; b_g^{-1}]=\sigma^2
\end{aligned}
\qquad\qquad\qquad\quad
\begin{aligned}[c]
\,
\end{aligned}
\end{equation}
\end{theorem}

These generators are depicted in Figure \ref{fig:gens_B2}. Here, the generators $a_i$ and $b_j$  are pure braids coming from the representation of the topological surface $\Sigma_g$ as a polygon of $4g$ sides with the standard identification of the edges, whereas $\sigma$ is a non-pure braid exchanging the two points $p_1$, $p_2 \in \Sigma_g$. For both $a_i$ and $b_j$, the only nontrivial string is the first one, which goes through the wall $\alpha_i$, respectively the wall $\beta_j$.
\begin{figure}[H]
\centering
\begin{tikzpicture}[
  scale=0.5,
  >=latex,
  line cap=round,
  line join=round,
  boundary/.style={thick,dashed},
  edge/.style={thick},
  looparc/.style={thick,blue,->},
  vertex/.style={circle,fill,inner sep=1.2pt},
  puncture/.style={circle,fill=red,inner sep=1.2pt}
]

\begin{scope}[shift={(-5.2,2.6)}]

  \coordinate (p1) at (-1.25,0);
  \coordinate (p2) at ( 1.25,0);

  \coordinate (A1) at ({ 2+sqrt(3)},1);
  \coordinate (A2) at ({ 1+sqrt(3)},{1+sqrt(3)});
  \coordinate (A3) at ({ 1+sqrt(3)/2},{1.5+sqrt(3)});
  \coordinate (A4) at (1,{2+sqrt(3)});
  \coordinate (A5) at (0,{2+sqrt(3)});
  \coordinate (A6) at (-1,{2+sqrt(3)});
  \coordinate (A7) at ({-1-sqrt(3)/2},{1.5+sqrt(3)});
  \coordinate (A8) at ({-1-sqrt(3)},{1+sqrt(3)});
  \coordinate (A9) at ({-1.5-sqrt(3)},{1+sqrt(3)/2});
  \coordinate (A10) at ({-2-sqrt(3)},1);
  \coordinate (A11) at ({-2-sqrt(3)},-1);

  \draw[boundary] (A1) -- (A2);
  \draw[edge,->] (A2) -- (A3) -- (A4) -- (A5);
  \draw[edge]    (A5) -- (A6) -- (A7);
  \draw[edge,<-] (A7) -- (A8) -- (A9) -- (A10);
  \draw[boundary] (A10) -- (A11);

  \foreach \X in {A2,A4,A6,A8,A10}
    \node[vertex] at (\X) {};

  \node[vertex]   at (p1) {};
  \node[puncture] at (p2) {};
  \draw[looparc] (p1) -- (A3);
  \draw[looparc] (A7) -- (p1);

  \node[below] at (p1) {$p_1$};
  \node[below] at (p2) {$p_2$};
  \node[below] at (0,-1) {$a_i$};

  \node[above right] at (A3) {$\alpha_i$};
  \node[above]       at ($(A5)!0.5!(A6)$) {$\beta_i$};
  \node[above left]  at (A7) {$\alpha_i$};
  \node[above left]  at (A9) {$\beta_i$};

\end{scope}

\begin{scope}[shift={(5.2,5.2)}]

  \coordinate (q1) at (-1.25,0);
  \coordinate (q2) at ( 1.25,0);

  \coordinate (B1)  at ({-2-sqrt(3)},1);
  \coordinate (B2)  at ({-2-sqrt(3)},-1);
  \coordinate (B3)  at ({-1.5-sqrt(3)},{-1-sqrt(3)/2});
  \coordinate (B4)  at ({-1-sqrt(3)},{-1-sqrt(3)});
  \coordinate (B5)  at ({-1-sqrt(3)/2},{-1.5-sqrt(3)});
  \coordinate (B6)  at (-1,{-2-sqrt(3)});
  \coordinate (B7)  at (0,{-2-sqrt(3)});
  \coordinate (B8)  at (1,{-2-sqrt(3)});
  \coordinate (B9)  at ({1+sqrt(3)/2},{-1.5-sqrt(3)});
  \coordinate (B10) at ({1+sqrt(3)},{-1-sqrt(3)});
  \coordinate (B11) at ({2+sqrt(3)},-1);

  \draw[boundary] (B1) -- (B2);
  \draw[edge,->] (B2) -- (B3) -- (B4) -- (B5);
  \draw[edge]    (B5) -- (B6) -- (B7) -- (B8) -- (B9);
  \draw[edge,<-] (B9) -- (B10);
  \draw[boundary] (B10) -- (B11);

  \foreach \X in {B2,B4,B6,B8,B10}
    \node[vertex] at (\X) {};

  \node[vertex]   at (q1) {};
  \node[puncture] at (q2) {};
  \draw[looparc] (q1) -- (B5);
  \draw[looparc] (B9) -- (q1);

  \node[above] at (q1) {$p_1$};
  \node[above] at (q2) {$p_2$};
  \node[above] at (0,1) {$b_j$};

  \node[below left]  at (B3) {$\alpha_j$};
  \node[below left]  at (B5) {$\beta_j$};
  \node[below]       at (B7) {$\alpha_j$};
  \node[below right] at (B9) {$\beta_j$};

\end{scope}

\begin{scope}[shift={(0,-2.8)}, scale=0.92]

  \coordinate (s1) at (-1.25,0);
  \coordinate (s2) at ( 1.25,0);

  \coordinate (S1)  at ({ 2+sqrt(3)},1);
  \coordinate (S2)  at ({ 1+sqrt(3)},{1+sqrt(3)});
  \coordinate (S3)  at (1,{2+sqrt(3)});
  \coordinate (S4)  at (-1,{2+sqrt(3)});
  \coordinate (S5)  at ({-1-sqrt(3)},{1+sqrt(3)});
  \coordinate (S6)  at ({-2-sqrt(3)},1);
  \coordinate (S7)  at ({-2-sqrt(3)},-1);
  \coordinate (S8)  at ({-1-sqrt(3)},{-1-sqrt(3)});
  \coordinate (S9)  at (-1,{-2-sqrt(3)});
  \coordinate (S10) at (1,{-2-sqrt(3)});
  \coordinate (S11) at ({1+sqrt(3)},{-1-sqrt(3)});
  \coordinate (S12) at ({2+sqrt(3)},-1);

  \draw[boundary]
    (S1) -- (S2) -- (S3) -- (S4) -- (S5) -- (S6) --
    (S7) -- (S8) -- (S9) -- (S10) -- (S11) -- (S12) -- cycle;

  \node[vertex] at (s1) {};
  \node[vertex] at (s2) {};

  \draw[thick,blue,->] (s1) to[out=90,in=90] (s2);
  \draw[thick,red,->]  (s2) to[out=270,in=270] (s1);

  \node[left]  at (s1) {$p_1$};
  \node[right] at (s2) {$p_2$};
  \node[above right] at (0,1) {$\sigma$};

\end{scope}

\end{tikzpicture}

\caption{Generators of $\mathsf{B}_2(\Sigma_g)$}\label{fig:gens_B2}
\end{figure}
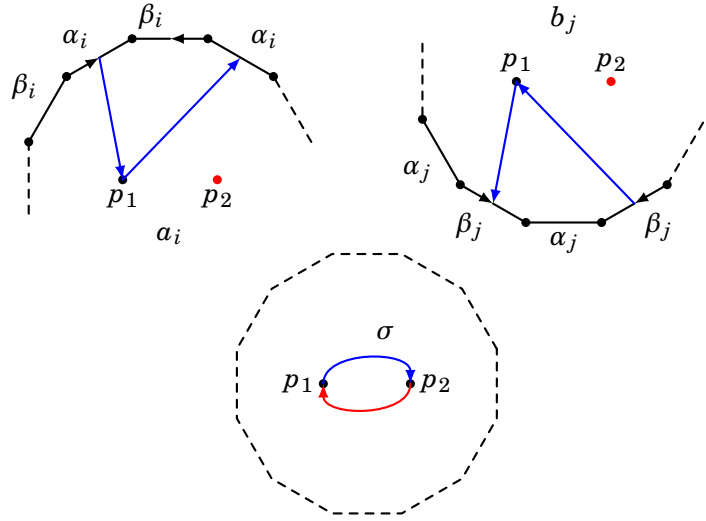

\begin{remark} \label{rmk:meaning_sigma}
In terms of the isomorphism on the right in \eqref{eq:iso-braids}, the generator $\sigma$ corresponds to the homotopy class in $\mathrm{Sym}^2 \Sigma_g-\delta$ of a topological loop  that ``winds once around the diagonal $\delta$".
\end{remark}

\section{Admissible epimorphisms and their geometric interpretation: branched Galois covers of  $\operatorname{Sym}^2 \Sigma_g$}
\label{sec:admissible_epi_Galois_covers}

\subsection{Admissible epimorphisms} \label{subsec:admissible_epimorphisms}

\begin{definition} \label{def:admissible_epimorphism}
Let $g, \, n \geq 2$ be integers, and let $G$ be a finite group. A group epimorphism
\begin{equation} \label{eq:admissible_epimorphism}
\varphi \colon \B_2(\Sigma_g) \to G
\end{equation}
is called an \emph{admissible epimorphism} of type $(g, \, n)$ if $\varphi(\sigma)$ has order $n$.  In this case,  $G$ will be called an \emph{admissible $($full$)$ braid quotient of type} $(g, \, n)$. 
\end{definition}

\begin{remark} \label{rmk:admissible_epimorphism}
The inclusion map  $\iota \colon \operatorname{Sym}^2  \Sigma_g - \delta \to  \operatorname{Sym}^2  \Sigma_g$ induces a group epimorphism
\begin{equation}
\iota_* \colon \B_2(\Sigma_g) \to \pi_1(\operatorname{Sym}^2 \Sigma_g, \, \mathscr{P}),
\end{equation}
whose kernel is the normal closure of the subgroup generated by $\sigma$. Thus, a group epimorphism $\varphi \colon \B_2(\Sigma_g) \to G$ is admissible if and only if $\varphi$ does not factor through $\pi_1(\operatorname{Sym}^2 \Sigma_g, \, \mathscr{P})$. 
\end{remark}

Using Remark \ref{rmk:meaning_sigma}, together with the Grauert-Remmert extension theorem and Serre's GAGA, we see that the existence of an admissible group epimorphism $\varphi \colon \B_2(\Sigma_g) \to G$ of type $(g, \, n)$  is equivalent to the existence of a Galois cover
\begin{equation}
f \colon S \to \operatorname{Sym}^2 \Sigma_g,
\end{equation}
with Galois group $G$ and branched precisely along $\delta$ with branching order $n$, see \cite[Section 1]{Pol18}. Note that $S$ is a smooth projective surface, since $\operatorname{Sym}^2 \Sigma_g$ is so and $\delta \subset \operatorname{Sym}^2 \Sigma_g$ is a smooth divisor.
\begin{remark}
 The invariants of $\operatorname{Sym}^2 \Sigma_g$ are 
\begin{equation} \label{eq:invariants_Sym}
\begin{split}
K_{\Sym^2\Sigma_g}^2 & =(g-1)(4g-9) \\
\chi_{\mathrm{top}}(\Sym^2\Sigma_g) &=(g-1)(2g-3) \\
\chi(\mathcal{O}_{\Sym^2\Sigma_g}) &=\frac{(g-1)(g-2)}{2}\\
p_g(\Sym^2\Sigma_g)&=\frac{g(g-1)}{2} \\
q(\Sym^2\Sigma_g)&=g, 
\end{split}
\end{equation}
see \cite{MacDon62}. The birational structure of $\Sym^2 \Sigma_g$ is also well-known:
\begin{itemize}
\item if $g \geq 3$ then $\Sym^2 \Sigma_g$ is a minimal surface of general type, and for non-hyperelliptic $\Sigma_g$  it does not contain any rational curve;
\item if $g=2$ then $\Sym^2 \Sigma_2$ is the blow-up of the Jacobian $J(\Sigma_2)$ at one point. The exceptional divisor $E \subset \Sym^2 \Sigma_2$ corresponds to the unique  hyperelliptic $g^1_2$ in $\Sigma_2$, and it intersects the diagonal $\delta$ transversally at the six Weierstrass points. 
\end{itemize}
\end{remark}

We can now compute the invariants of $S$.
\begin{proposition} \label{prop:invariants_S}
Let $f \colon S \to \Sym^2 \Sigma_g$ be the $G$-cover associated with an admissible epimorphism $\varphi \colon \B_2(\Sigma_g) \to G$ of type $(g, \, n)$.  Then $S$ is a minimal surface of general type with
\begin{equation} \label{eq:invariants_S}
\begin{split}
K_{S}^2&=c_1^2(S)=|G|(g-1)(4g-9+12\mathfrak{n}-4\mathfrak{n}^2) \\
\chi_{\operatorname{top}}(S)&=c_2(S)=|G|(g-1)(2g-3+2\mathfrak{n}) \\
\chi(\mathcal{O}_S)&=\frac{1}{12}|G|(g-1)(6g-12+14\mathfrak{n}-4\mathfrak{n}^2),
\end{split}
\end{equation}
where we set $\mathfrak{n} = 1 - \frac{1}{n}$.
\end{proposition}
\begin{proof}
Since $f \colon S \to \Sym^2\Sigma_g$ is a $G$-cover branched with order $n$ over the diagonal $\delta \subset \Sym^2\Sigma_g$, if $R \subset S$ is the ramification divisor then $f^*\delta=nR$, and the restriction $f \mid_R \colon R \to \delta$  is unramified of degree $\frac{|G|}{n}$. Since $\delta$ is a smooth curve of genus $g$, we have  $\chi_{\operatorname{top}}(\delta)=2-2g$,  and so
\begin{equation}
\chi_{\operatorname{top}}(R)=\chi_{\operatorname{top}}(\delta)\frac{|G|}{n}
=(2-2g)\frac{|G|}{n}.
\end{equation}
Since $S \setminus R\rightarrow\Sym^2\Sigma_g\setminus\delta$ is unramified of degree $|G|$, using  the additivity of $\chi_{\operatorname{top}}$ and the second equality in \eqref{eq:invariants_Sym} we get
\begin{align*}
\chi_{\operatorname{top}}(S)&=\chi_{\operatorname{top}}(S\setminus R)+\chi_{\operatorname{top}}(R)\\
&=|G|\,\chi_{\operatorname{top}}(\Sym^2\Sigma_g\setminus\delta)+(2-2g)\frac{|G|}{n}\\
&=|G|\,\left(\chi_{\operatorname{top}}(\Sym^2\Sigma_g)-\chi_{\operatorname{top}}(\delta)\right)+(2-2g)\frac{|G|}{n}\\
&=|G|\,\left((g-1)(2g-3)-(2-2g)\right)+(2-2g)\frac{|G|}{n}\\
&=|G|\,\left( (g-1)(2g-3)+(2g-2)\left(1-\frac{1}{n} \right) \right)\\
&=|G|\, (g-1)(2g-3+2\mathfrak{n}).
\end{align*}
Now, recall that 
\begin{equation}
\delta^2=4-4g, \quad K_{\operatorname{Sym}^2\Sigma_g}\cdot\delta=6(g-1),
\end{equation}
 the latter equality following from the adjunction formula. Using the Hurwitz formula and the first equality in \eqref{eq:invariants_Sym}, we obtain
\begin{align*}
K_S^2&=\left( f^* K_{\Sym^2\Sigma_g}+(n-1)R\right)^2\\
&=\left(f^* K_{\Sym^2\Sigma_g}\right)^2+2 (n-1) \,f^* K_{\Sym^2\Sigma_g}\cdot R+(n-1)^2 R^2\\
&=|G|\, K_{\Sym^2\Sigma_g}^2+\frac{2(n-1)}{n}f^*K_{\Sym^2\Sigma_g}\cdot f^*\delta+\frac{(n-1)^2}{n^2}\left(f^*\delta\right)^2\\
&=|G|\, (g-1)(4g-9)+2\frac{(n-1)}{n}|G|\, K_{\Sym^2\Sigma_g}\cdot\delta+\left(\frac{n-1}{n}\right)^2|G|\cdot\delta^2\\
&=|G|\, \left((g-1)(4g-9)+2\left(1-\frac{1}{n}\right)6(g-1)+\left(1-\frac{1}{n}\right)^2 (4-4g)\right)\\
&=|G|\, (g-1)(4g-9+12\mathfrak{n}-4\mathfrak{n}^2).
\end{align*}
The value of $\chi(\mathcal{O}_S)$ can now be computed by applying  Noether's formula, see \cite[I.14]{Be96}. 

In order to prove that $S$ is a minimal surface, we will show that it contains no  rational curves. Let us  distinguish two cases. 
\begin{itemize}
\item  If $\Sigma_g$ is non-hyperelliptic, then $\operatorname{Sym}^2\Sigma_g$ is a surface of general type which contains no rational curves. Since $f \colon S \to \operatorname{Sym}^2\Sigma_g$ is a finite morphism, the same holds for $S$.
\item If $\Sigma_g$ is hyperelliptic, then $\operatorname{Sym}^2\Sigma_g$  contains a unique rational curve $E$, which is the smooth curve  determined by the unique $g_2^1$ on $\Sigma_g$.   The curve $E$ meets the
diagonal $\delta$ transversely at exactly the $2g+2$ points
corresponding to the Weierstrass points of $\Sigma_g$.
Thus the inverse image $f^*(E)$ is a smooth, possibly disconnected curve.

 If  $D$ is any of the irreducible components of $f^*(E)$, the induced cover $D\to E\simeq\mathbb P^1$ is unramified outside
$E\cap\delta$ and has ramification order $n$ at each of the $2g+2$ points of $E \cap \delta$.  If  $d $ is the degree of the restricted cover $D\to E$,  using  the Hurwitz formula and $n \geq 2$ we get 
\begin{equation}
2g(D)-2 = d  \left(-2 + (2g+2)\left(1 - \frac{1}{n} \right) \right) \geq d(g-1) >0,
\end{equation}
hence   $g(D) \geq 2$. Since any rational curve in  $S$  would map onto  $E$,  it follows that  $S$ contains no rational curves. 
\end{itemize}
Finally, since $K_S^2 >0$ and $q(S) \geq q(\Sym^2 \Sigma_g)=g$, it follows that $S$ is a surface of general type, see \cite[Proposition X.1]{Be96}.     
\end{proof}

\subsection{Chern slope and signature}

Given a minimal surface of general type $S$, we define its \emph{Chern slope} as
\begin{equation}
\nu(S) = K_S^2/\chi_{\operatorname{top}}(S) = c_1^2(S)/c_2(S).
\end{equation} 
If  $S$  has maximal Albanese dimension, we have
\begin{equation}
\nu(S) \geq 1/2
\end{equation}
and equality holds if and only if  $q(S)=2$ and the canonical model of  $S$
is a double cover of the Albanese surface branched on an ample divisor with at most negligible singularities \cite{BaPaSt16}.

For every minimal surface of general type, we also have
\begin{equation}
\nu(S) \leq 3
\end{equation}
and equality holds if and only if  the universal cover of $S$ is the unit ball $\mathbb{B}^2 \subset \mathbb{C}^2$, see 
\cite[Chapter VII]{BHPV03}.

It is rather difficult to construct surfaces $S$ with 
\begin{equation} \label{eq:arctic-zone}
2 < \nu(S) \leq 3.
\end{equation} 
Indeed, for a long time only a few examples with Chern pairs satisfying \eqref{eq:arctic-zone} were known, and the corresponding region of the plane $(c_1^2, \, c_2)$ was called the \emph{arctic zone} \cite[p. 5]{Per81}.

The \emph{signature} $\operatorname{sign}(S)$ of  $S$ is defined as the signature of the intersection form on the middle cohomology group $H^2(S, \, \mathbb{R})$. Using  the Hirzebruch signature formula and Proposition \ref{prop:invariants_S}, we get
\begin{equation}
\operatorname{sign}(S) = \frac{1}{3}\left( c_1^2(S) - 2c_2(S)\right) = \frac{1}{3}\,|G| \,(g-1) \left( 1 - \frac{4}{n^2} \right). 
\end{equation}
This shows that $S$ is in the arctic zone if and only if  $\operatorname{sign}(S) >0$ and that, for our surfaces, this is in turn equivalent to the numerical condition $n \geq 3$.   

Regarding the slope, we have 
\begin{proposition} \label{prop:slope_S}
Let $f \colon S \to \Sym^2 \Sigma_g$ be the $G$-cover associated with an admissible epimorphism $\varphi \colon \B_2(\Sigma_g) \to G$ of type $(g, \, n)$.   Let $\nu(S) = \nu(g, \, n)$ be the Chern slope of $S$. Then
\begin{equation}
2 \leq \nu(S) \leq  \frac{239}{102} \simeq 2.343137.
\end{equation}
 Moreover
\begin{itemize}
\item the minimum $2$ occurs precisely when $n=2$,  independently of  $g;$
\item the upper bound  $239/102$ can be attained only when $(g, \, n)=(2, \, 12)$.
\end{itemize} 
 \end{proposition}
\begin{proof}
Using Proposition \ref{prop:invariants_S}, we obtain
\begin{equation} \label{eq:nu(g,n)}
\nu(g, \, n) = 2 + \frac{n^2-4}{(2g-1)n^2-2n} \geq  2 ,
\end{equation}
and equality holds if and only if  $n=2$.
In order to obtain an upper bound, we observe that, for fixed $n>2$, the fraction in \eqref{eq:nu(g,n)} is strictly decreasing in $g$, whereas for $n=2$ it vanishes identically. So we get
\begin{equation}
\nu(g, \, n) \leq \nu(2, \, n) = 2 + \frac{n^2-4}{3n^2-2n}.  
\end{equation}
For $x  \geq 2$, the real function $u(x)=(x^2-4)/(3x^2-2x)$ attains its global maximum at $x_0 = 6 + 4 \sqrt{2} \simeq 11.657$. Since $n$ is an integer, it follows that the global maximum of  $u(n)$ is attained at either $n=11$ or  $n=12$. Computing explicitly the two values
\begin{equation}
u(11) = \frac{117}{341}, \quad u(12) = \frac{35}{102},
\end{equation}
we see that the maximum of  $u(n)$ occurs at $n=12$, and we are done.
\end{proof}

\begin{remark} \label{rmk:slope_maximizing}
It is an interesting problem to construct, and possibly to classify, slope-maximizing admissible covers $\varphi \colon \B_2(\Sigma_2) \to G$ of type $(2, \, 12)$, but we will not address it here. Nevertheless, we will construct infinitely many examples with $\nu(S) > 2$ (namely, with positive signature) in Section \ref{sec:_extraspecial_quotients}, by means of  extra-special quotients. See in particular Proposition \ref{prop:slope_S_extra-special}. 
\end{remark}

\subsection{Generating vectors} \label{subsec:generating_vectors}
We now define a combinatorial object that completely encodes the datum of an admissible epimorphism.

\begin{definition} \label{def:generating_vector}
Let $g, \, n \geq 2$ be integers and let $G$ be a finite group. A \emph{generating vector of type} $(g, \, n)$ on $G$ is an ordered  $(2g+1)$-tuple of  elements  of $G$
\begin{equation} \label{eq:generating_vector}
\mathscr{V}=(\a_1,\, \b_1, \ldots, \a_g, \, \b_g, \, \s),
\end{equation}
generating $G$  and with $o(\s)=n$, obtained as the image of the ordered $(2g+1)$-tuple
\begin{equation}
(a_1, \, b_1, \ldots, a_g, \, b_g, \, \sigma) 
\end{equation}
via an admissible epimorphism $\varphi \colon \mathsf{B}_2(\Sigma_g) \to G$. 
\end{definition} 

By definition, the classification of admissible full braid quotients is equivalent to the classification of finite groups that admit generating vectors. By Theorem \ref{thm:Bellingeri}, an ordered $(2g+1)$-tuple of elements generating  $G$, with $o(\mathsf{s})=n$, is a generating vector of type $(g, \, n)$ if and only if its entries satisfy the following relations:
\begin{equation} \label{B2} \tag{B2}
\begin{aligned}[c]
[\a_i, \; \s^{-1}\a_i \s^{-1}]&=1\\
[\b_i, \; \s^{-1}\b_i\s^{-1}]&=1
\end{aligned}
\qquad\qquad
\begin{aligned}[c]
1\leq i\leq g
\end{aligned}
\end{equation}
\begin{equation}\label{B3} \tag{B3}
\begin{aligned}[c]
[\a_i, \; \s^{-1}\a_j\s]&=1\qquad\\
[\a_i, \; \s^{-1}\b_j\s]&=1\\
[\b_i, \; \s^{-1}\a_j\s]&=1\\
[\b_i, \; \s^{-1}\b_j\s]&=1
\end{aligned}
\qquad\quad
\begin{aligned}[c]
j<i
\end{aligned}
\end{equation}
\begin{equation} \label{B4} \tag{B4}
\begin{aligned}[c]
[\a_i, \; \s^{-1}\b_i \s^{-1}]&=\s^2
\end{aligned}
\qquad\;\;\:
\begin{aligned}[c]
1\leq i \leq g
\end{aligned}
\end{equation}
\begin{equation} \label{TR} \tag{TR}
\begin{aligned}[c]
[\a_1, \; \b_1^{-1}] \cdots [\a_g,\; \b_g^{-1}]=\s^2
\end{aligned}
\qquad\qquad\qquad\quad
\begin{aligned}[c]
\,
\end{aligned}
\end{equation}

\begin{example} \label{ex:generating vector_G abelian}
If $G$ is abelian  then the only surviving relations are 
\eqref{B4} and \eqref{TR}, both giving $\s^2=1$. 
\end{example}

More generally, we have

\begin{proposition} \label{prop_a_b_central_s^2}
If either $\a_i \in Z(G)$ or $\b_i \in Z(G)$ for some $i$, then $\s^2=1$. 
\end{proposition}
\begin{proof}
If $\a_i\in Z(G)$ for some $i$, then \eqref{B4} immediately gives $\s^2=1$. Assume now $\b_i\in Z(G)$. Then
\begin{equation}
\s^{-1}\b_i\s^{-1}=\b_i\s^{-2},
\end{equation}
and so \eqref{B4} yields
\begin{equation}
[\a_i, \, \s^{-2}]=\s^2.
\end{equation}
Thus  $\a_i\s^{-2}\a_i^{-1}=1$, which implies $\s^2=1$.
\end{proof}
\begin{remark} \label{rmk:action_auto}
The automorphism group $\operatorname{Aut}(G)$ acts freely on the set of generating vectors. So, if $N_g(G)$ is the total number of admissible epimorphisms $\varphi \colon \mathsf{B}_2(\Sigma_g) \to G$, and $N^{\circ}_g(G)$ is the corresponding number up to automorphisms of $G$, then 
\begin{equation} \label{eq:N-N^circ}
N^{\circ}_g(G) = \frac{1}{|\operatorname{Aut}(G)|}N_g(G).
\end{equation}
\end{remark}

\section{Cyclic admissible quotients of $\mathsf{B}_2(\Sigma_g)$} \label{sec:cyclic_quotients}

Let $\C_{m}=\langle x \mid x^{m}=1\rangle$ be the cyclic group of order $m$. Since $\C_m$ is abelian, if  $\mathscr{V} = (\a_1,\b_1,\ldots,\a_g,\b_g,\s)$ is a generating vector for $\C_m$ then $\s$ is an element of order 2, see Example \ref{ex:generating vector_G abelian}. Hence, $m$ must be even  and $\s=x^{m/2}$.

\begin{proposition}
\label{prop:cyclic_generators}
Let $\C_{m}=\langle x \mid x^{m}=1\rangle$ be the cyclic group of order $m$, with $m$ even. Then, for every $k\geq 1$, the number of ordered $k$-tuples $(z_1,\ldots, z_k)$ of elements of  $ \C_{m}$ such that
\begin{equation*}
\langle z_1,\ldots,z_k, \, x^{m/2}\rangle=\C_{m}
\end{equation*}
equals $2^kJ_k(m/2).$
\end{proposition}
\begin{proof}
Consider the quotient map $\C_{m}\to \C_{m}/\langle x^{m/2}\rangle \simeq \C_{m/2}, \ z\mapsto \bar{z}.$
Since $x^{m/2}$ generates the kernel of this map, we immediately see that for $z_1,\dotso, z_k\in \C_{m}$ we have
\begin{equation}
\label{eq:cyclic_quotient_equivalence}
\langle z_1,\ldots,z_k,\, x^{m/2}\rangle=\C_{m}
\quad \text{if and only if}\quad
\langle \bar z_1,\ldots,\bar z_k\rangle=\C_{m/2}.
\end{equation}
By Proposition \ref{prop:Jordan_totient_1}, the number of ordered $k$-tuples $(\bar z_1,\ldots,\bar z_k)$ generating $\C_{m/2}$ is $J_k(m/2)$, and every such $k$-tuple has precisely
$2^k$ lifts, which have the form 
\begin{equation}
( z_1 x^{\alpha_1 m/2},\ldots ,z_k x^{\alpha_k m/2}), \quad \alpha_i\in \{0, \, 1\}.
\end{equation}
Thus, the claim follows by \eqref{eq:cyclic_quotient_equivalence}.
\end{proof}

\begin{theorem}
\label{thm:cyclic}
The number $N_g(\C_m)$ of admissible epimorphisms $\varphi\colon \B_2(\Sigma_g)\longrightarrow \C_m$ is given by
\begin{equation}
\label{eq:number_of_cyclic}
N_g(\C_m)=
\begin{cases}
0 & \emph{if } \; m \; \emph{is odd}\\[2mm]
2^{2g}J_{2g}\left(m/2 \right)
&  \emph{if } \; m \; \emph{is even}.
\end{cases}
\end{equation}
\end{theorem}
\begin{proof}
By the discussion above, we already know that $N_g(\C_m)=0$ if $m$ is odd. Next, assume that $m$  is even and let $\mathscr{V} = (\a_1,\b_1,\ldots,\a_g,\b_g,\s)$ be a $(2g+1)$-tuple generating $\C_{m}$. Example
\ref{ex:generating vector_G abelian} implies that $\mathscr{V}$ is a generating vector for $\C_{m}$
 if and only if  $\s=x^{m/2}$. The computation of $N_g(\C_{m})$ now follows from Proposition \ref{prop:cyclic_generators} applied with $k=2g$.
\end{proof}

\begin{remark}
Using \eqref{eq:Jordan_totient_2}, we can rewrite
\eqref{eq:number_of_cyclic} as follows
\begin{equation}
\label{eq:number_of_cyclic_alternative}
N_g(\C_m)=
\begin{cases}
0 &  \text{if } \; m \; \text{is odd} \\[2mm]
m^{2g}
\displaystyle\prod_{p\mid (m/2)}
\left(1-\dfrac{1}{p^{2g}}\right)
&  \text{if } \; m \; \text{is  even}.
\end{cases}
\end{equation}
\end{remark}

\smallskip

Finally, by Remark \ref{rmk:action_auto} we obtain
\begin{corollary}
\label{cor:cyclic_up_to_auto}
The number $N^{\circ}_g(\mathsf{C}_m)$ of admissible epimorphisms $\varphi\colon \B_2(\Sigma_g)\longrightarrow
\C_m$, up to automorphisms of $\mathsf{C}_m$, is
\begin{equation}
N^{\circ}_g(\mathsf{C}_m) = \frac{1}{\phi(m)} N_g(\mathsf{C}_m).
\end{equation}
\end{corollary}

\section{Dihedral admissible quotients of $\mathsf{B}_2(\Sigma_g)$} \label{sec:_dihedral_quotients}

Let   $m \geq 2$ be an integer, and let  $G=\D_{2m}$ be  the dihedral group of order $2m$, presented as 
\begin{equation}
\D_{2m} = \langle x, \, y \; | \; x^2=1, \, y^m=1, \, xyx^{-1}=y^{-1} \rangle. 
\end{equation}
If $m\geq 3$, interpreting $\D_{2m}$ as the group of isometries of a regular $m$-gon, we call the $m$ elements 
\begin{equation}
1,\, y, \, \ldots, y^{m-1}
\end{equation}
the \emph{rotations} of $\D_{2m}$, and the $m$ involutions
\begin{equation}
x, \, xy, \, \ldots, xy^{m-1}
\end{equation}
the \emph{reflections} of $\D_{2m}$. The case $m=2$ is the only one where $\D_{2m}$ is abelian, in fact $\D_{4} \simeq \C_2 \times \C_2$. The nontrivial elements of $\D_4$ are three involutions and, in this special situation, we will not make any distinction between rotations and reflections. 

\begin{remark} \label{rmk:Dihedral_CCT}
Assume $m \geq 3$. Since $\langle y \rangle \simeq \C_{m}$ is a normal subgroup of index $2$,  it follows that $\mathsf{D}_{2m}$ is a CCT-\emph{group}, namely, a non-abelian group such that commutativity is a transitive relation in the set of noncentral elements, see \cite[Proposition 2.5]{PolSab22}. In other words, if $u_1, \, u_2, \, u_3 \in \D_{2m} - Z(\D_{2m})$ and $[u_1, \, u_2]=[u_2, \, u_3]=1$, then $[u_1, \, u_3]=1$.
\end{remark}

If $m=2$, then $\D_{4} \simeq \C_2 \times \C_2$ and so $\operatorname{Aut}(\D_4) \simeq \mathsf{S}_3$. If $m \geq 3$, then $\operatorname{Aut}(\D_{2m})$ consists of the automorphisms $$f_{ij} \colon \D_{2m} \to \D_{2m}, \qquad f_{ij}(x)=xy^i, \qquad f_{ij}(y)=y^j,$$
where $0 \leq i \leq m-1$ and $1 \leq j \leq m-1$ with $\operatorname{gcd}(j, \, m)=1$.
Thus,
\begin{equation} \label{eq:auto_dihedral}
|\operatorname{Aut}(\D_{2m})| =
\begin{cases}
6 & \text{if } \, m=2 \\
m \phi(m) & \text{if } \, m \geq 3.
\end{cases}
\end{equation}

\subsection{The case $m$ odd.}

We start by considering the case where $m$ is odd. Then $Z(\mathsf{D}_{2m})=\{1\}$, and the conjugacy classes of $\mathsf{D}_{2m}$ are the following:
\begin{equation} 
\{1\}, \; \; \{y, \,y^{-1} \}, \; \; \{y^2, \, y^{-2}\},  \ldots , \{ y^{(m-1)/2}, \, y^{(1-m)/2} \}, \; \; \{x, \, xy, \, \ldots, xy^{m-1} \}.
\end{equation}
So all the $m$ reflections are conjugate, whereas the rotation $y^i$ is only conjugate with the inverse rotation $y^{-i}$. Moreover,  we have
\begin{equation} \label{eq:centralizer-Dm-m-odd}
\begin{split}
& C_G(y^i)  = \langle y \rangle \simeq  \C_m \quad \text{for} \; \; 1 \leq  i \leq m-1 \\
& C_G(xy^i)  = \langle xy^i \rangle \simeq \C_2 \quad \text{for} \; \; 0 \leq i \leq m-1.
\end{split}
\end{equation}

\begin{lemma} \label{lem:s-reflection-m-odd}
Let $m\geq 3$ be an odd integer and let $\varphi \colon \B_2(\Sigma_g) \to \D_{2m}$ be an admissible epimorphism, with generating vector $(\a_1, \, \b_1, \, \ldots, \a_g, \, \b_g, \,\s)$.  Then $\s$ is a reflection.
\end{lemma}
\begin{proof}
By definition of admissible epimorphism, the element $\s$ is nontrivial; so, we must only show that $\s^2=1$.  If either $\a_i=1$ or $\b_i=1$ for some $i$, the claim follows from Proposition \ref{prop_a_b_central_s^2}. We can therefore assume that all the $\a_i$ and all the $\b_i$ are nontrivial, and so noncentral. Using relations \eqref{B3} we get
\begin{equation}
[\a_i, \, \s^{-1}\a_{i-1} \s]=[\b_i, \, \s^{-1} \a_{i-1} \s]=1 \quad \mathrm{for} \;\; i>1
\end{equation}
and so, by Remark \ref{rmk:Dihedral_CCT}, we have 
\begin{equation} \label{eq:m-odd-1}
[\a_i, \, \b_i] = 1 \quad \mathrm{for} \;\; i>1.
\end{equation}
Using relations \eqref{B3} we also get
\begin{equation}
[\a_{j+1}, \, \s^{-1}\a_j \s]=[\a_{j+1}, \, \s^{-1} \b_j \s]=1 
\end{equation}
and so, again by Remark \ref{rmk:Dihedral_CCT}, we obtain
\begin{equation}
[\s^{-1} \a_j \s, \; \s^{-1} \b_j \s]=1 \quad \mathrm{for} \; \; j<g, 
\end{equation}
that is 
\begin{equation} \label{eq:m-odd-2}
[\a_j, \, \b_j]=1 \quad \mathrm{for} \; \; j<g. 
\end{equation}
Combining \eqref{eq:m-odd-1} and \eqref{eq:m-odd-2} we get 
\begin{equation} 
[\a_i, \, \b_i]=1 \quad \text{for all} \; \; i \in \{1, \ldots, g\},
\end{equation}
and so relation \eqref{TR} yields $\s^2=1$, as claimed.
\end{proof}

\begin{lemma} \label{lem:ai-bj-rotations-m-odd}
Let $m \geq 3$ be an odd integer and let $\varphi \colon \B_2(\Sigma_g) \to \D_{2m}$ be an admissible epimorphism, with generating vector $(\a_1, \, \b_1, \, \ldots, \a_g, \, \b_g, \,\s)$.  Then $\a_i$, $\b_i$ are $($possibly trivial$)$ rotations for all $i$.
\end{lemma}
\begin{proof}
Assume, for the sake of contradiction, that $\a_i$ is a reflection. By Lemma \ref{lem:s-reflection-m-odd} we know that $\s$ is also a reflection, hence $\s^{-1}=\s$ and from \eqref{B2} we get $[\a_i, \, \s \a_i \s^{-1}]=1$. Since any conjugate of a reflection is a reflection,  the only possibility is $\a_i = \s \a_i \s^{-1}$, that is $\a_i \in C_G(\s) = \langle \s \rangle$, and so $\a_i=\s$. Using \eqref{B3} and \eqref{B4} we infer
\begin{equation}
\a_j, \, \b_j \in C_G(\s) \quad \textrm{for all} \; \;  j,
\end{equation}
and from this we deduce
\begin{equation}
\D_{2m} = \langle \a_1, \, \b_1, \, \ldots, \a_g, \, \b_g, \,\s \rangle \subseteq C_G(\s) \simeq \C_2,
\end{equation}
which is absurd. The case where $\b_i$ is a reflection can be excluded in the same way.
\end{proof}

\begin{proposition} \label{prop:dihedral-m-odd}
Let $m\geq 3$ be an odd integer. Then the number of admissible epimorphisms $\varphi \colon \mathsf{B}_2(\Sigma_g) \to \D_{2m}$ equals $mJ_{2g}(m)$.
\end{proposition}
\begin{proof}
Observe first that any $(2g+1)$-tuple $\mathscr{V}= (\a_1, \, \b_1, \, \ldots, \a_g, \, \b_g, \,\s)$, where the first $2g$ elements are rotations generating $\langle y \rangle \simeq \C_m$ and $\s$ is a reflection, is a generating vector. Indeed, one clearly has $\D_{2m}=\langle \a_1, \, \b_1, \, \ldots, \a_g, \, \b_g, \,\s \rangle$, and moreover $\mathscr{V}$ fulfills $\eqref{B2}$, \eqref{B3}, \eqref{B4} and \eqref{TR} since $\s^2=1$ and every two rotations commute.

Now let $\mathscr{V}= (\a_1, \, \b_1, \, \ldots, \a_g, \, \b_g, \,\s)$ be the generating vector associated with an admissible epimorphism $\varphi \colon \mathsf{B}_2(\Sigma_g) \to \D_{2m}$. By Lemma \ref{lem:s-reflection-m-odd} we know that $\s$ is a reflection. On the other hand, Lemma \ref{lem:ai-bj-rotations-m-odd}, together with the surjectivity of  $\varphi$, implies that, once we fix $\s$, the ordered $2g$-tuple $(\a_1, \, \b_1, \, \ldots, \a_g, \, \b_g)$ generates $\langle y \rangle \simeq \C_m$. 
By Proposition \ref{prop:Jordan_totient_1}, the number of such generating $2g$-tuples  is $J_{2g}(m)$; since $\s$ can be any of the $m$ reflections, it follows that the total number of generating vectors for $\D_{2m}$ is $m J_{2g}(m)$.
\end{proof}

\subsection{The case $m$ even.}

We now consider the case where the integer $m$ is even. Assume $m\geq 4$. Then $Z(\mathsf{D}_{2m})=\langle y^{m/2}\rangle \simeq \C_2$, and the conjugacy classes of $\mathsf{D}_{2m}$ are the following:
\begin{equation} 
\begin{split}
& \{1\}, \; \; \{y, \,y^{-1} \}, \; \; \{y^2, \, y^{-2}\},  \ldots , \{ y^{m/2-1 }, \, y^{1-m/2} \}, \\ 
& \{y^{m/2} \}, \\
& \{x, \, xy^2, \, \ldots, xy^{m-2} \}, \; \; \{xy, \, xy^3, \, \ldots, xy^{m-1} \}. 
\end{split}
\end{equation}
There are two differences with respect to the case where $m$ is odd:
\begin{itemize}
\item the presence of the conjugacy class $\{y^{m/2} \}$ of the unique nontrivial central element;
\item the partition of the $m$ reflections into two distinct classes, each containing $m/2$ elements. 
\end{itemize}  
Consequently, the centralizers are:
\begin{equation} \label{eq:centralizer-Dm-m-even}
\begin{split}
&C_G(y^i)  = \langle y \rangle \simeq \C_m \quad  \text{for} \;\;  0  < i \leq m-1 \; \; \text{and} \; \; i \neq m/2, \\
&C_G(y^{m/2})  = \D_{2m}, \\
&C_G(xy^i)  = \langle xy^i, \, y^{m/2}\rangle \simeq \C_2 \times  \C_2 \quad \text{for} \;\; 0 \leq i \leq m-1.
\end{split}
\end{equation}
We can now prove the analogue of Lemma \ref{lem:s-reflection-m-odd} in this situation.
\begin{lemma} \label{lem:s-reflection-m-even}
Let $m\geq 4$ be an even integer and let $\varphi \colon \B_2(\Sigma_g) \to \D_{2m}$ be an admissible epimorphism, with generating vector $(\a_1, \, \b_1, \, \ldots, \a_g, \, \b_g, \,\s)$.  Then $\s$ is a reflection.
\end{lemma}
\begin{proof}
We first observe that $\s$ is not a central element. Otherwise, relations  \eqref{B3} and \eqref{B4} would imply
\begin{equation}
\begin{split}
[\a_i, \, \a_j]&=[\b_i, \, \b_j]= [\a_i, \, \b_j]=[\b_i, \, \a_j]=1 \quad \text{for}\;\; j<i\\
[\a_i, \, \b_i]&=1 \quad \text{for all} \;\; i,
\end{split}
\end{equation} 
and so  $\D_{2m}=\langle \a_1,\, \b_1, \ldots, \a_g, \,\b_g, \, \s \rangle$ would be commutative, a contradiction. Thus, we must only show that $\s^2=1$.  If either $\a_i$ or $\b_i$ is a central element for some $i$,  this follows from Proposition  \ref{prop_a_b_central_s^2}. Otherwise, we can exploit the fact that $\D_{2m}$ is a CCT-group, exactly as in the proof of Lemma \ref{lem:s-reflection-m-odd}.
\end{proof}
We now want to determine the possible values of  $\a_1, \, \b_1, \ldots  ,\a_g, \, \b_g$. It turns out that they are always rotations (cf. Lemma \ref{lem:ai-bj-rotations-m-odd}), except possibly in one case. 

\begin{lemma} \label{lem:ai-bi-reflection-n=2-4}
Under the same assumptions of Lemma \emph{\ref{lem:s-reflection-m-even}}, if  $\a_i$ or $\b_i$ is a reflection for some $i$, then $m =4$. 
\end{lemma}
\begin{proof}
Assume that $\a_i$ is a reflection for some $i$, say $\a_i=xy^k$ (the case where $\b_i$ is a reflection can be treated in the same way). By Lemma \ref{lem:s-reflection-m-even} we may furthermore assume that, up to automorphisms of $\D_{2m}$, we have $\s=x$. Then relation $\eqref{B2}$ becomes
\begin{equation}
xy^k x xy^k x = x xy^k x xy^k,
\end{equation}
hence $y^{-2k} = y^{2k}$, which means $(y^{2k})^2=1$. Now there are two possibilities.
\medskip

\emph{Case 1}. $y^{2k}=1$. This implies $k \in \{0, \, m/2 \}$, hence
\begin{equation}\label{Case-1-a_i-b_i-reflection}
\text{either} \quad \a_i=  \s  \a_i \s=x  \quad \text{or} \quad \a_i=  \s  \a_i \s= xy^{m/2}.
\end{equation}
Using \eqref{B3} and \eqref{B4}, we deduce
\begin{equation}
\begin{split}
[\a_i, \; \s\a_j\s] &= [\a_i, \; \s\b_j\s] =1 \quad \text{for} \; \; j<i \\
[\a_j, \; \s\a_i\s] &= [\b_j, \; \s\a_i\s] =1 \quad \text{for} \; \; j>i \\
[\s \a_i\s, \; \b_i]& =1,
\end{split}
\end{equation}
and so, in both cases, \eqref{Case-1-a_i-b_i-reflection} yields
\begin{equation}
\a_j, \, \b_j \in C_G(\a_i) = C_G(\s \a_i \s)= \langle x, \, y^{m/2} \rangle \simeq \C_2 \times \C_2 \quad \text{for all } j.
\end{equation}
This implies
\begin{equation}
\D_{2m} = \langle \a_1, \, \b_1, \, \ldots, \a_g, \, \b_g, \,\s \rangle  \subseteq \langle x, \, y^{m/2} \rangle \simeq   \C_2 \times \C_2, 
\end{equation}
which is impossible because  $m \geq 4$.
\medskip

\emph{Case 2}. $y^{2k}=y^{m/2}$. This is possible only if $4 \mid m$,  in which case $k\in \{m/4, \, 3m/4\}$. Thus  
\begin{equation}\label{Case-2-a_i-b_i-reflection}
\langle\mathsf a_i,\mathsf s\mathsf a_i\mathsf s\rangle
=\langle xy^{m/4},y^{m/2}\rangle
\simeq\mathsf C_2\times\mathsf C_2.
\end{equation}
Using \eqref{B3} and \eqref{B4} as above,  \eqref{Case-2-a_i-b_i-reflection} yields 
    \begin{equation}
    \a_j, \, \b_j \in C_G(\a_i) = C_G(\s \a_i \s)= \langle   xy^{m/4}, \, y^{m/2} \rangle \simeq \C_2 \times \C_2 \quad \text{for all } j.
    \end{equation}
    This implies 
    \begin{equation}
    \D_{2m}  = \langle \a_1, \, \b_1, \, \ldots, \a_g, \, \b_g, \,\s \rangle
     = \langle   xy^{m/4}, \, y^{m/2}, \, x \rangle = \langle y^{m/4}, \, x \rangle \simeq \D_8,
    \end{equation}
    hence $m=4$.
\end{proof}

We are now in a position to state the analogue of Proposition \ref{prop:dihedral-m-odd} when $m$ is even.

\begin{proposition} \label{prop:dihedral-m-even}
Let $m\geq 2$ be an even integer. Then the number of admissible epimorphisms $\varphi \colon \mathsf{B}_2(\Sigma_g) \to \D_{2m}$ equals $\varepsilon \cdot mJ_{2g}(m)$, where 
\begin{equation} \label{eq:epsilon}
\varepsilon =
\begin{cases}
1 &\text{ if } \, m \notin \{2, \, 4\} \\
2 &\text{ if } \, m=4\\
3 \cdot 2^{2g-1} & \text{ if } \, m=2.
\end{cases}
\end{equation}
\end{proposition}

\begin{proof}
Assume first $m \notin \{2, \, 4\}$ and let $\mathscr{V}= (\a_1, \, \b_1, \, \ldots, \a_g, \, \b_g, \,\s)$ be the generating vector associated with an admissible epimorphism $\varphi \colon \mathsf{B}_2(\Sigma_g) \to \D_{2m}$. Then Lemma \ref{lem:ai-bi-reflection-n=2-4} shows that all the elements $\a_i$ and $\b_i$ are rotations; on the other hand, by Lemma \ref{lem:s-reflection-m-even} we know that $\s$ is a reflection. Hence, by the same argument used in the proof of  Proposition \ref{prop:dihedral-m-odd}, we can conclude that the number of generating vectors is $m J_{2g}(m)$.

 It remains to consider the cases $m=2$ and $m=4$.
\smallskip

\emph{Case $m=2$}. We have
\begin{equation}
\D_{4} \simeq \C_2 \times \C_2 =  \langle x, \, y \; \; | \; \; x^2=y^2= [x, \, y]=1 \rangle.
\end{equation}
Since $\D_4$ is abelian, a $(2g+1)$-tuple  $\mathscr{V}= (\a_1, \, \b_1, \, \ldots, \a_g, \, \b_g, \,\s)$ is a generating vector if and only if  its entries generate $\mathsf D_4$  and  $\mathsf s\neq 1$.
Up to automorphisms, we can take $\s=x$, and so we are asking that the elements $\a_1, \, \b_1, \ldots, \a_g, \, \b_g, \, x$ generate the whole group. This is the same as requesting that the $2g$-tuple $(\a_1, \, \b_1, \ldots, \a_g, \, \b_g)$ contains at least one element not belonging to the set $\{1, \, x\}$. Since the total number of $2g$-tuples in $\D_4$ is $4^{2g}$ and those containing only the elements $1, \, x$ are $2^{2g}$, we get $4^{2g} - 2^{2g} = 2^{2g} (2^{2g} -1)$ occurrences. This number must be multiplied by $3$ since we have three possibilities for $\s$, hence the total number of generating vectors is 
\begin{equation}
3 \cdot 2^{2g} (2^{2g} -1) = 3 \cdot 2^{2g-1} \cdot 2J_{2g}(2).
\end{equation} 

\smallskip

\emph{Case $m=4$}. We have
\begin{equation}
\D_{8} =  \langle x, \, y \; \; | \; \; x^2=y^4=1, \;  xyx^{-1}= y^{-1} \rangle.
\end{equation}
By Lemma \ref{lem:s-reflection-m-even},  if $\mathscr{V} = (\a_1, \, \b_1, \, \ldots, \a_g, \, \b_g, \,\s)$ is a generating vector for $\D_8$, then $\s$ is a reflection. Hence, by the same argument used in the proof of  Proposition \ref{prop:dihedral-m-odd}, the number $4J_{2g}(4)$ counts the number of generating vectors $(\a_1, \, \b_1, \, \ldots, \a_g, \, \b_g, \,\s)$ for $\D_8$ such that 
\begin{equation}
\a_i, \; \b_i \in \langle y \rangle \simeq \C_4 \quad \text{for all } i. 
\end{equation}  
However, as we have seen in the proof of Lemma \ref{lem:ai-bi-reflection-n=2-4}, in the case $m=4$ we also have another possibility, namely (assuming $\s=x$, up to automorphisms)
 \begin{equation} \label{eq:D8}
\a_i, \; \b_i \in C_G(xy)=\langle xy, \, y^2 \rangle \simeq \C_2 \times \C_2 \quad \text{for all } i. 
\end{equation}  
It is immediate to check that, if a $(2g+1)$-tuple $\mathscr{V}=(\a_1, \, \b_1, \ldots, \a_g, \, \b_g, \, x)$ fulfills \eqref{eq:D8}, then all the relations \eqref{B2}, \eqref{B3}, \eqref{B4}, \eqref{TR} are satisfied. Thus, such a tuple  $\mathscr{V}$ is a generating vector if and only if it generates $\D_8$. This is equivalent to requiring that the $2g$-tuple $(\a_1, \, \b_1, \ldots, \a_g, \, \b_g)$ contains at least one element not belonging to the set $\{1, \, y^2\}$.  Since the total number of $2g$-tuples in $C_G(xy)$ is $4^{2g}$ and those containing only the elements $1, \, y^2$ are $2^{2g}$, we get $4^{2g} - 2^{2g} = 2^{2g} (2^{2g} -1)$ occurrences. This number must be multiplied by $4$, since $\s$ may be any reflection of $\D_8$, obtaining 
\begin{equation}
4(4^{2g} - 2^{2g}) = 4^{2g+1} \left( 1- \frac{1}{2^{2g}} \right) = 4J_{2g}(4)
\end{equation} 
occurrences, which count the number of generating vectors $(\a_1, \, \b_1, \, \ldots, \a_g, \, \b_g, \,\s)$ for $\D_8$ such that $\a_i$ or $\b_i$ is a reflection for some $i$. Summing up, the total number of generating vectors for $\D_8$ is 
\begin{equation}
4J_{2g}(4)+ 4J_{2g}(4)=2 \cdot 4J_{2g}(4).
\end{equation}
\end{proof}

Putting together Propositions \ref{prop:dihedral-m-odd} and \ref{prop:dihedral-m-even}, we finally get

\begin{theorem} \label{thm:dihedral}
Let $m\geq 2$ be an integer. Then the number $N_g(\D_{2m})$ of admissible epimorphisms $\varphi \colon \mathsf{B}_2(\Sigma_g) \to \D_{2m}$ is given by 
\begin{equation} \label{eq:epsilon_1}
N_g(\D_{2m}) = \varepsilon \cdot mJ_{2g}(m), \quad 
\mathrm{where } \; \;  \varepsilon =
\begin{cases}
1 &\text{if } \, m \notin \{2, \, 4\} \\
2 &\text{if } \, m=4 \\
3 \cdot 2^{2g-1} & \text{if } \, m=2.
\end{cases}
\end{equation}
\end{theorem}

Finally, by Remark \ref{rmk:action_auto} and \eqref{eq:auto_dihedral}, we obtain
\begin{corollary} \label{cor:dihedral_up_to_auto}
The number $N^{\circ}_g(\mathsf{D}_{2m})$ of admissible epimorphisms $\varphi \colon \B_2(\Sigma_g) \to \D_{2m}$, up to automorphisms of $\mathsf{D}_{2m}$, is 
\begin{equation}
N^{\circ}_g(\mathsf{D}_{2m}) = \frac{1}{|\operatorname{Aut}(\D_{2m})|}N_g(\mathsf{D}_{2m}) = 
\begin{cases}
J_{2g}(m)/\phi(m) & \text{if } \, m \notin\{2, \, 4\} \\
J_{2g}(4) & \text{if } \, m=4 \\
2^{2g-1} J_{2g}(2) & \text{if } \, m=2.
\end{cases}
\end{equation}
\end{corollary}

\section{Extra-special admissible quotients of $\mathsf{B}_2(\Sigma_g)$} \label{sec:_extraspecial_quotients}

For the following classical notion, see \cite[p. 183]{Gor07} and \cite[p. 123]{Is08}.

\begin{definition} \label{def:extra-special}
Let $p$ be a prime number. A finite $p$-group $G$ is said to be \emph{extra-special} if its center $Z(G)$ is cyclic of order $p$ and the factor group $G/Z(G)$ is a nontrivial elementary abelian $p$-group.
\end{definition}

An elementary abelian $p$-group may be regarded as a finite-dimensional
vector space over the finite field $\mathbb{F}_p$. In particular, it has the form
$V \simeq \mathbb{F}_p^{\, \dim V}$, and $G$ therefore arises as a central
extension
\begin{equation} \label{eq:extension-extra}
  1 \longrightarrow \C_p \longrightarrow G \longrightarrow V \longrightarrow 1.
\end{equation}
Since $V$ is abelian, it follows that the derived subgroup of $G$
is contained in $\C_p$.  As $G$ is non-abelian, we necessarily have
$[G, \, G]=\C_p$, so the commutator subgroup coincides with the center. In particular, every extra-special $p$-group is nilpotent of class $2$.


\begin{remark}\label{rmk:symplectic-form}
If $G$ is extra-special, then $Z(G)=\langle z \rangle\simeq \mathbb{F}_p$, and  we can define a bilinear form $V \times V \to \mathbb{F}_p$ as follows: for every $v_1, \, v_2 \in V$, we set $(v_1, \, v_2)=k$, where $[g_1, \, g_2]=z^k$ and $g_i$ is any lift of  $v_i$ in $G$. This turns out to be a symplectic form on $V$, hence
$\dim V$ is even and $|G|=p^{\dim V +1}$ is an odd power of $p$.
\end{remark}

The simplest examples of extra-special groups are the non-abelian groups of order $p^3$. 
For every prime number $p$, there are precisely two isomorphism classes
$M(p)$, $N(p)$ of such groups, namely
\begin{equation*}
  \begin{split}
    M(p)& = \langle x, \, y, \, z \; | \; x^p=y^p=1, \;
    z^p=1,  \; [x, \,
    z]=[y,\, z]=1, \;  [x, \, y]=z^{-1} \rangle \\
    N(p)& = \langle x, \, y, \, z \; | \; x^p=y^p=z, \;
    z^p=1,  \; [x, \,
    z]=[y,\, z]=1, \;  [x, \, y]=z^{-1} \rangle, \\
  \end{split}
\end{equation*}
see \cite[Theorem 5.1 of
Chapter 5]{Gor07}. For odd primes $p$, the groups $M(p)$ and $N(p)$ are distinguished by their exponents, equal to $p$ and $p^2$, respectively. If $p=2$, then $M(2) \simeq \D_8$ and  $N(2) \simeq \mathsf{Q}_8$.
\medskip

The symplectic structure associated with an extra-special group imposes strong restrictions on its occurrence as an admissible quotient of $\mathsf{B}_2(\Sigma_g)$. In fact, apart from one exceptional case, the order of such a quotient is determined by the genus $g$ and the prime $p$. This is the content of the following result, which will play a key role in the subsequent classification.

\begin{proposition} \label{prop:m=g}
Let $g\geq 2 $ be a fixed integer, and let $G$ be an extra-special $p$-group of order $p^{2m+1}$. Suppose that there exists an admissible epimorphism 
\begin{equation}
\varphi\colon \B_2(\Sigma_g) \to G.
\end{equation}
Then there are two cases:
\begin{itemize}
\item[$\boldsymbol{(a)}$] $m=g;$
\item[$\boldsymbol{(b)}$] $p=2, \; \; m=1, \; \;  G \simeq \mathsf{D}_8$.
\end{itemize}
\end{proposition}
\begin{proof}
Using the notation above, we write
\begin{equation}
V=G/Z(G) \simeq \mathbb{F}_p^{2m}, \quad Z(G)= \langle z \rangle \simeq \mathsf{C}_p
\end{equation}
and we denote by $(\cdot, \, \cdot)$ the symplectic form on $V$ induced by commutators.   

Let  $\mathscr{V}=(\a_1, \, \b_1, \, \ldots, \a_g, \, \b_g, \,\s)$ be the generating vector corresponding to the admissible epimorphism $\varphi$, and let 
\begin{equation} \label{eq:barV}
(\bar{\a}_1, \, \bar{\b}_1, \, \ldots, \bar{\a}_g, \, \bar{\b}_g, \,\bar{\s})
\end{equation}
be the  $(2g+1)$-tuple of corresponding images in $V$. 
Since all commutators in $G$ are central,  by relation \eqref{TR} we get  $\s^2 \in Z(G)$, and so $\s^2=z^k$, with $k \in \{0, \ldots, p-1\}$. Relations \eqref{B3} and \eqref{B4} then imply 
\begin{equation} \label{eq:Gram}
 (\bar{\a}_i, \, \bar{\a}_j) = (\bar{\b}_i, \, \bar{\b}_j) =0, \quad  (\bar{\a}_i, \,\bar{\b}_j) = k\delta_{ij}.
  \end{equation}
Indeed, conjugation in a nilpotent group of class  $2$ changes an element only by a central factor, which disappears when taking commutators.
  
Now, there are two possibilities.
\begin{itemize}
\item $k \neq 0$. Then the $2g$ elements  $\bar{\a}_1, \, \bar{\b}_1, \, \ldots, \bar{\a}_g, \, \bar{\b}_g$ have a non-degenerate Gram matrix, hence they are linearly independent in $V$. On the other hand, the surjectivity of $\varphi$ implies that the entries of \eqref{eq:barV} generate  $V$. So we get
\begin{equation}
2g \leq \dim V = 2m \leq 2g+1,
\end{equation}
that is, $m=g$. This is  case $\boldsymbol{(a)}$.

\item $k=0$. In this case $\s^2=1$, hence $p=2$. Moreover, \eqref{eq:Gram} now implies that
\begin{equation}
W = \operatorname{span}(\bar{\a}_1, \, \bar{\b}_1, \, \ldots, \bar{\a}_g, \, \bar{\b}_g)
\end{equation}
is a totally isotropic subspace of $V$. Since $\dim V = 2m$, it follows $\dim W \leq m$. On the other hand,  the surjectivity of $\varphi$ yields 
\begin{equation}
V = W + \operatorname{span}(\bar{\s}),
\end{equation}
hence
\begin{equation}
2m = \dim V \leq \dim W + 1 \leq m+1.
\end{equation}
Thus $m=1$, hence $\dim V=2$ and $|G|=8$. There are exactly two extra-special groups of order $8$, namely $\mathsf{D}_8$ and $\mathsf{Q}_8$. Now observe that, since $W$ is a proper subspace of $V$, we have $\s \notin Z(G)$, i.e.,  $\s$ is a noncentral involution of $G$. No such involution exists in $\mathsf{Q}_8$, hence $G \simeq \mathsf{D}_8$. This is case $\boldsymbol{(b)}$.
\end{itemize}
\end{proof}

The case $G \simeq \mathsf{D}_8$ was already investigated in the previous section. In fact,  $\mathsf{D}_8$ is an admissible quotient of $\B_2(\Sigma_g)$ for all $g \geq 2$, see Theorem \ref{thm:dihedral}. Thus, in the sequel we only need to consider the case of extra-special quotients of $\B_2(\Sigma_g)$ having order   $p^{2g+1}$.
\medskip

We need the following classification of extra-special $p$-groups, which uses $M(p)$ and $N(p)$ as building blocks; see  \cite[Section 5 of Chapter 5]{Gor07} and \cite[Section 3]{Pol22}.

\begin{proposition} \label{prop:extra-special-groups}
  If $g \geq 2$ is an integer and $p$ is a prime number, there
  are exactly two isomorphism classes of extra-special $p$-groups
  of order
  $p^{2g+1}:$
  \begin{itemize}
    \item the central product $\mathsf{H}_{2g+1}(\mathbb{F}_p)$
      of $g$ copies
      of $M(p)$, having presentation
      \begin{equation} \label{eq:H5}
	\begin{split}
	  \mathsf{H}_{2g+1}(\mathbb{F}_p)
	  = \langle \, & x_1, \,
	  y_1,
	  \ldots, x_g,\, y_g,
	  \, z  \; |  \; x_{i}^p =
	  y_{i}^p=z^p=1, \\
	  & [x_{i}, \, z]  =
	  [y_{i}, \, z]= 1, \\
	  & [x_i, \, x_j]=
	  [y_i, \, y_j] =
	  1, \\
	  & [x_{i}, \,y_{j}]
	  =z^{- \delta_{ij}} \, \rangle.
	\end{split}
      \end{equation}
      If $p$ is odd, this group has exponent $p$ and is isomorphic to the
      matrix Heisenberg group $\mathcal{H}_{2g+1}(\mathbb{F}_p) \subset
      \mathsf{GL}(g+2, \, \mathbb{F}_p)$ of dimension $2g+1$ over the
      field $\mathbb{F}_p;$
    \item the central product $\mathsf{G}_{2g+1}(\mathbb{F}_p)$
      of $g-1$ copies
      of $M(p)$ and one copy of $N(p)$,  having presentation
      \begin{equation} \label{eq:G5}
	\begin{split}
	  \mathsf{G}_{2g+1}(\mathbb{F}_p)
	  = \langle \, & x_1, \,
	  y_1,
	  \ldots, x_g,\, y_g,
	  \, z \; | \;  x_{g}^p =
	  y_{g}^p=z, \\
	  & x_{1}^p = y_{1}^p=
	  \ldots = x_{g-1}^p =
	  y_{g-1}^p=z^p=1, \\
	  & [x_{i}, \, z]  =
	  [y_{i}, \, z]= 1, \\
	  & [x_i, \, x_j]=
	  [y_i, \, y_j] =
	  1, \\
	  & [x_{i}, \,y_{j}]
	  =z^{- \delta_{ij}} \, \rangle.
	\end{split}
      \end{equation}
      If $p$ is odd, this group has exponent $p^2$.
  \end{itemize}
\end{proposition}

\begin{remark}\label{rmk:symplectic-basis}
The presentations \eqref{eq:H5} and \eqref{eq:G5} show  that, in both cases $\mathsf{H}_{2g+1}(\mathbb{F}_p)$ and $\mathsf{G}_{2g+1}(\mathbb{F}_p)$, we get
\begin{equation}
(\bar{x}_i, \, \bar{x}_j)=0, \quad (\bar{y}_i, \, \bar{y}_j)=0, \quad (\bar{x}_i, \, \bar{y}_j)= - \delta_{ij}
\end{equation}
for all $i, \, j \in \{1, \ldots, g\}$, thus 
\begin{equation} \label{eq:ordered-symplectic-basis}
\bar{x}_1, \, \bar{y}_1, \ldots, \bar{x}_g, \, \bar{y}_g
\end{equation}
is an ordered symplectic basis for $V=G/Z(G)$. If $\phi \in \operatorname{Aut}(G)$, then $\phi$ induces an invertible linear map $\bar{\phi} \in \mathsf{GL}(V)$. Such a map preserves the symplectic form on $V$, i.e., it belongs to $\mathsf{Sp}(V) \simeq \mathsf{Sp}(2g, \, \mathbb{F}_p)$, if and only if $\phi$ acts  trivially on $Z(G) = \langle z \rangle$.
\end{remark}
The previous remark leads to the following description of $\operatorname{Aut}(G)$, which can be found in \cite[Theorem 1]{Win72}.

\begin{proposition} \label{prop:H/I-extraspecial}
Let us denote by $H$ the subgroup of $\operatorname{Aut}(G)$ acting trivially on $Z(G)$. Then $\operatorname{Aut}(G) = \langle \theta \rangle \cdot H$, where $\theta$ has order $p-1$, $H \cap \langle \theta \rangle = \langle 1  \rangle$ and $H/ \operatorname{Inn}(G)$ is isomorphic to a subgroup of  $\mathsf{Sp}(2g, \, \mathbb{F}_p)$. More precisely$:$
\begin{itemize}
\item if $p$ is odd and $G=\mathsf{H}_{2g+1}(\mathbb{F}_p)$, then $H/\operatorname{Inn}(G)$ is isomorphic to $\mathsf{Sp}(2g, \, \mathbb{F}_p);$
\item if $p$ is odd and $G=\mathsf{G}_{2g+1}(\mathbb{F}_p)$, then 
$H/\operatorname{Inn}(G)$ is isomorphic to the semidirect product of $\mathsf{Sp}(2g-2, \, \mathbb{F}_p)$ and a normal extra-special group of order $p^{2g-1}$ $($if $g=1$ then  $H/\operatorname{Inn}(G)$ has order $p);$
\item if $p=2$, then $H=\operatorname{Aut}(G)$ and
\begin{equation}
H/\operatorname{Inn}(G) = \operatorname{Out}(G) \simeq
\begin{cases}
\mathsf{O}^{+}(2g,\mathbb{F}_2)
& \text{if }\, G=\mathsf{H}_{2g+1}(\mathbb{F}_2)\\
\mathsf{O}^{-}(2g,\mathbb{F}_2)
& \text{if } \,  G=\mathsf{G}_{2g+1}(\mathbb{F}_2),
\end{cases}
\end{equation}
where $\mathsf{O}^{+}$ and $\mathsf{O}^{-}$ denote the
orthogonal groups of nondegenerate quadratic forms
of plus and minus type, respectively.

\end{itemize}
\end{proposition}
\medskip

\subsection{The case $p$ odd}

First, we consider the case where $p$ is an odd prime. 
\begin{corollary} \label{cor:Aut-extraspecial}
Let $G$ be an extra-special group of order $p^{2g+1}$, with $p$ odd and $g \geq 2$.
\begin{itemize}
\item If $ G=\mathsf{H}_{2g+1}(\mathbb{F}_p)$, then
\begin{equation}
|\operatorname{Aut}(G)|= (p-1) \, p^{g(g+2)} \prod_{i=1}^g  (p^{2i} -1 ).
\end{equation}
\item If $ G=\mathsf{G}_{2g+1}(\mathbb{F}_p)$, then
\begin{equation}
|\operatorname{Aut}(G)|= (p-1) \, p^{g(g+2)} \prod_{i=1}^{g-1} (p^{2i} -1).
\end{equation}
\end{itemize}
\end{corollary}

\begin{proof}
Since $\operatorname{Inn}(G)\simeq G/Z(G)$, Proposition \ref{prop:H/I-extraspecial} implies that 
\begin{itemize}
\item if $ G=\mathsf{H}_{2g+1}(\mathbb{F}_p)$ then
\begin{equation}
|\operatorname{Aut}(G)|= (p-1) \, p^{2g}  \cdot |\mathsf{Sp}(2g, \, \mathbb{F}_p)|;
\end{equation}
\item if $ G=\mathsf{G}_{2g+1}(\mathbb{F}_p)$ then
\begin{equation}
|\operatorname{Aut}(G)|= (p-1) \, p^{4g-1}  \cdot |\mathsf{Sp}(2g-2, \, \mathbb{F}_p)|.
\end{equation}
\end{itemize}
The claim now follows, because for every $n \geq 1$ we have $|\mathsf{Sp}(2n, \, \mathbb{F}_p)| = p^{n^2} \prod _{i=1}^n (p^{2i} - 1)$.
\end{proof}

\begin{lemma} \label{lem:s-central-extraspecial}
Let $G$ be an extra-special group of order $p^{2g+1}$, with $p$ \emph{odd} and $g\geq 2$. Let $\varphi \colon \mathsf{B}_2(\Sigma_g) \to G$ be an admissible epimorphism, with generating vector $(\a_1, \, \b_1, \, \ldots, \a_g, \, \b_g, \,\s)$.  Then $\s \in Z(G) = \langle z \rangle$.
\end{lemma}
\begin{proof}
Let us consider the image $\bar{\s}$ of $\s$ in $V=G/Z(G) \simeq \mathbb{F}_p^{2g}$. Relation \eqref{TR} implies $\s^2 \in [G, \, G] = Z(G)$, hence $2 \bar{\s} = 0$. Since $p$ is odd, this means $\bar{\s}=0$, that is, $\s \in Z(G)$.
\end{proof}

\begin{proposition} \label{prop:extra-special_congruence}
Let $G$ be an extra-special group of order $p^{2g+1}$, with $p$ \emph{odd} and $g\geq 2$. Then $G$ is an admissible quotient of $\B_2(\Sigma_g)$ if and only if $g \equiv -1 \pmod{p}$. 
\end{proposition}
\begin{proof}
Let $\mathscr{V}=(\a_1,\, \b_1, \ldots, \a_g, \, \b_g, \, \s)$ be a generating vector corresponding to an admissible epimorphism $\varphi \colon \B_2(\Sigma_g) \to G$. By Lemma \ref{lem:s-central-extraspecial} we have $\s \in Z(G)$, thus relation \eqref{B4} gives $[\a_i, \, \b_i] = \s^2$, and so 
\begin{equation} \label{eq:inverse}
\begin{split}
[\a_i, \, \b_i^{-1}]& =\a_i \b_i^{-1} \a_i^{-1} \b_i = \b_i^{-1} [\a_i, \, \b_i]^{-1} \b_i 
= \b_i^{-1} \s^{-2} \b_i = \s^{-2}.
\end{split}
\end{equation}
Therefore, relation \eqref{TR} yields $\s^{-2g} = \s^2$. Since $o(\s)=p$, we deduce $-2g \equiv 2 \pmod{p}$ and so, since $p$ is odd, $g \equiv -1 \pmod{p}$.  This shows the first implication.

In order to prove the other implication, let us set
\begin{equation}
\a_i:=x_i, \quad \b_i:=y_i, \quad \s :=z^{(p-1)/2}
\end{equation}
for all $i=1,\ldots, g$ (note that $\s$ is well-defined since $p$ is odd). Straightforward computations show that the $(2g+1)$-tuple $\mathscr{V}=(\a_1,\, \b_1, \ldots, \a_g, \, \b_g, \, \s)$ satisfies the relations \eqref{B2} and \eqref{B3}. Moreover, if $g \equiv -1 \pmod{p}$, then $\mathscr{V}$ fulfills the relations \eqref{B4} and \eqref{TR} as well. Hence, we produced a generating vector for $G$, and the proof is complete. 
\end{proof}

\begin{theorem} \label{thm:extraspecial}
Let $p$ be  an odd prime number and $g \geq 2$ be an integer such that $g \equiv -1 \pmod{p}$. If $G$ is an extra-special group of order $p^{2g+1}$, then the number $N_g(G)$ of admissible epimorphisms $\varphi \colon \B_2(\Sigma_g) \to G$ is given by
\begin{equation} \label{eq:number_N_extraspecial_1}
N_g(G) =  (p-1) \, p^{g(g+2)} \prod_{i=1}^g (p^{2i} -1).
\end{equation}
In particular, $N_g(G)$ is independent of $G$.
\end{theorem}
 \begin{proof}
Let $\varphi \colon \mathsf{B}_2(\Sigma_g) \to G$ be an admissible epimorphism, with generating vector $(\a_1, \, \b_1, \, \ldots, \a_g, \, \b_g, \,\s)$; since $p$ is odd, we have $\s \in Z(G)$ (Lemma \ref{lem:s-central-extraspecial}). Up to automorphisms of $G$, we can assume $\s = z^{(p-1)/2}$, hence
\begin{equation} 
 [\a_i, \, \a_j] = [\b_i, \, \b_j] =1, \quad  [\a_i, \, \b_j] = z^{-\delta_{ij}}
  \end{equation}
for all $i, \, j \in \{1, \ldots, g\}$. Consequently, the elements $\bar{\a}_i, \, \bar{\b}_j \in V = G/Z(G) \simeq \mathbb{F}_p^{2g}$ satisfy 
\begin{equation} 
 (\bar{\a}_i, \, \bar{\a}_j) = (\bar{\b}_i, \, \bar{\b}_j) =0, \quad  (\bar{\a}_i, \,\bar{\b}_j) = -\delta_{ij},
  \end{equation}
 namely, $\bar{\a}_1, \, \bar{\b}_1, \ldots, \bar{\a}_g, \, \bar{\b}_g$ is an ordered symplectic basis of $V$.  Conversely, given any ordered symplectic basis  $ \bar{\a}_1, \, \bar{\b}_1, \ldots, \bar{\a}_g, \, \bar{\b}_g$ of $V$, it can be lifted to precisely $p^{2g}$ generating vectors of $G$, which have the form
 \begin{equation}
 (\a_1z^{\alpha_1}, \, \b_1 z^{\beta_1}, \ldots, \a_gz^{\alpha_g}, \, \b_g z^{\beta_g}, \, \s=z^{(p-1)/2}), \quad \alpha_i, \, \beta_j \in \{0, \ldots, p-1\}.
 \end{equation}
Therefore, the number of generating vectors of $G$ with $\s=z^{(p-1)/2}$ is  precisely $p^{2g}$ times the number of symplectic bases  of $V$. In order to obtain the total number of generating vectors, we must multiply further by the number of nontrivial elements in $Z(G)$, which is $p-1$. Summing up, we get
\begin{equation}
N_g(G) =  (p-1) \, p^{2g} \cdot |\mathsf{Sp}(2g, \, \mathbb{F}_p)|,
\end{equation}
 which is \eqref{eq:number_N_extraspecial_1}.
\end{proof}

Using Theorem \ref{thm:extraspecial}, together with Remark \ref{rmk:action_auto} and Corollary \ref{cor:Aut-extraspecial}, we deduce

\begin{corollary} \label{cor:extraspecial_auto}
Let $p$ be  an odd prime number and $g \geq 2$ be an integer such that $g \equiv -1 \pmod{p}$. If $G$ is an extra-special group of order $p^{2g+1}$, then the number $N^{\circ}_g(G)$ of admissible epimorphisms $\varphi \colon \B_2(\Sigma_g) \to G$, up to automorphisms of $G$, is given by
\begin{equation} \label{eq:number_N_extraspecial_auto}
N^{\circ}_g(G) = 
\begin{cases}
1 & \text{if } \, G=\mathsf{H}_{2g+1}(\mathbb{F}_p) \\
p^{2g} -1 & \text{if } \, G=\mathsf{G}_{2g+1}(\mathbb{F}_p).
\end{cases}
\end{equation}
\end{corollary}

It is remarkable that, in the case $G=\mathsf{H}_{2g+1}(\mathbb{F}_p) $, there exists precisely one admissible epimorphism $\varphi \colon \B_2(\Sigma_g) \to G$ up to automorphisms of $G$, namely, the one corresponding to the generating vector $(x_1, \, y_1, \ldots, x_g, \, y_g, \, z^{(p-1)/2})$.
\medskip

\subsection{The case $p=2$.}

Now we consider the case of extra-special admissible quotients of $\B_2(\Sigma_g)$ of order $2^{2g+1}$. If $p=2$ and $\bar{v} \in V=G/Z(G)$, then we
can set 
\begin{equation}
q(\bar{v})=c, \quad \text{where} \quad v^2=z^c \;\; \text{and} \;\;c \in \{0, \,
1\}.
\end{equation}
This is a quadratic form on $V \simeq \mathbb{F}_2^{2g}$. If $\bar{v}$  is expressed
in coordinates, with respect to the symplectic basis
\eqref{eq:ordered-symplectic-basis} of $V$, by the vector $(\xi_1, \, \psi_1, \ldots,
  \xi_g, \,
\psi_g) \in \mathbb{F}_2^{2g}$,  then a straightforward computation yields
\begin{equation} \label{eq:form-of-quadratic-forms}
  q(\bar{v})=
  \begin{cases}
    \xi_1\psi_1+\cdots+\xi_g \psi_g & \textrm{if }
    G=\mathsf{H}_{2g+1}(\mathbb{F}_2) \\
    \xi_1\psi_1+\cdots+\xi_g \psi_g+ \xi_g^2+\psi_g^2 &
    \textrm{if }
    G=\mathsf{G}_{2g+1}(\mathbb{F}_2).
  \end{cases}
\end{equation}
These are the two possible normal forms for a non-degenerate quadratic form
on  $\mathbb{F}_2^{2g}$; they have Arf invariant equal to $0$ and $1$, respectively, see for instance \cite{Dye78} or \cite[Chapter 10]{Li97}.\footnote{\label{footnote:Arf}There are several different (but related) notions of Arf invariant. If $V$ is a finite-dimensional $\mathbb{F}_2$-vector space and $q \colon V \to \mathbb{F}_2$ is a non-degenerate quadratic form, one can set 
\begin{equation*}
\operatorname{arf}(q):=\frac{1}{\sqrt{|V|}}\sum_{\alpha \in V} (-1)^{q(\alpha)}.
\end{equation*}
This is an invariant assuming values in $\{\pm 1\}$. The Arf invariant we use in this paper is the invariant $\operatorname{Arf}(q) \in \{0, \, 1\}$, related to the previous one by $\operatorname{arf}(q)=(-1)^{\operatorname{Arf}(q)}$.}

In both cases, the symplectic and the quadratic
forms are related by
\begin{equation}
  q(\bar{v} + \bar{w})=q(\bar{v})+q(\bar{w})+(\bar{v}, \, \bar{w})
  \quad
  \textrm{for all } \bar{v}, \, \bar{w} \in V.
\end{equation}
We have seen that, if $\phi \in \mathrm{Aut}(G)$,
then $\phi$ induces a linear map $\bar{\phi} \in \mathsf{GL}(V)$; moreover,
if $p=2$, then $\phi$ acts trivially on $Z(G)=[G, \, G] \simeq \mathsf{C}_2$,
and this in turn
implies that $\phi$ preserves the symplectic form
on $V$. In other words, if we identify $V$ with $\mathbb{F}_2^{2g}$
via the symplectic basis \eqref{eq:ordered-symplectic-basis}, we have $\bar{\phi}
\in \mathsf{Sp}(2g, \, \mathbb{F}_2)$.

We are now in a position to describe the structure of the automorphism group for an extra-special $2$-group,  see again \cite[Theorem 1]{Win72}. It turns out that $\bar{\phi}$ is actually an element of the orthogonal group of $V$ with respect to the quadratic form $q$.

\begin{proposition} \label{prop:Out(G)}
  Let $G$ be an extra-special group of order $2^{2g+1}$, with $g\geq 2$. Then the
  kernel of the group homomorphism $\mathrm{Aut}(G) \to \mathsf{Sp}(2g,
  \, \mathbb{F}_2)$ given by $\phi \mapsto \bar{\phi}$ is the subgroup
  $\mathrm{Inn}(G)$ of inner automorphisms of $G$. Therefore
  $\mathrm{Out}(G)
  = \mathrm{Aut}(G)/\mathrm{Inn}(G)$ embeds in $\mathsf{Sp}(2g,
  \, \mathbb{F}_2)$. More precisely, $\mathrm{Out}(G)$ coincides
  with the
  orthogonal group $\mathsf{O}_{\epsilon}(2g, \, \mathbb{F}_2)$,
  of order
  \begin{equation} \label{eq:order-orthogonal}
    |\mathsf{O}_{\epsilon}(2g, \,
    \mathbb{F}_2)|=2^{g(g-1)+1}(2^g-\epsilon)
    \prod_{i=1}^{g-1}(2^{2i}-1),
  \end{equation}
  associated with the quadratic form $\mathrm{\eqref{eq:form-of-quadratic-forms}}$. Here $\epsilon =
  1$ if $G=\mathsf{H}_{2g+1}(\mathbb{F}_2)$ and $\epsilon = -1$ if
  $G=\mathsf{G}_{2g+1}(\mathbb{F}_2)$.\footnote{Because of their relation with orthogonal geometry in characteristic $2$, the groups  $\mathsf{H}_{2g+1}(\mathbb{F}_2)$ and $\mathsf{G}_{2g+1}(\mathbb{F}_2)$ are also indicated by  $2_+^{2g+1}$ and $2_-^{2g+1}$, respectively.}
\end{proposition}

\begin{corollary} \label{cor:Aut-2-extraspecial}
  Let $G$ be an extra-special group of order $2^{2g+1}$, with $g\geq 2$. We have
  \begin{equation} \label{eq:Aut(G)}
    |\mathrm{Aut}(G)|=2^{g(g+1)+1}(2^g-\epsilon)
    \prod_{i=1}^{g-1}(2^{2i}-1)  = \frac{2^{g+1}}{2^g + \epsilon}\,\cdot \,  |\mathsf{Sp}(2g, \, \mathbb{F}_2)|.
       \end{equation}
\end{corollary}
\begin{proof}
  By Proposition \ref{prop:Out(G)} we get
  $|\mathrm{Aut}(G)|=|\mathrm{Inn}(G)|
  \cdot |\mathsf{O}_{\epsilon}(2g, \, \mathbb{F}_2)|$. Since
  $\mathrm{Inn}(G) \simeq G/Z(G)$ has order $2^{2g}$, the claim
  follows from \eqref{eq:order-orthogonal}.
\end{proof}

\begin{proposition} \label{prop:extra-special_2_congruence}
Let $G$ be an extra-special group of order $2^{2g+1}$, with $g\geq 2$.  Then $G$ is an admissible quotient of $\B_2(\Sigma_g)$ if and only if $g$ is odd.  In this case, for all generating vectors we have $\s^2=z$.
\end{proposition}
\begin{proof}
Let $\mathscr{V}=(\a_1,\, \b_1, \ldots, \a_g, \, \b_g, \, \s)$ be a generating vector corresponding to an admissible epimorphism $\varphi \colon \B_2(\Sigma_g) \to G$. Let us now look at the images of these elements in $V=G/Z(G) \simeq \mathbb{F}_2^{2g}$. Recalling the definition of the bilinear form $(\cdot, \, \cdot)$ and of the quadratic form $q$ on $V$, and using  $\s^2 \in Z(G) = \langle z \rangle$, from relations \eqref{B2}, \eqref{B3}, \eqref{B4} we infer
\begin{equation}
 (\bar{\a}_i, \, \bar{\a}_j) = (\bar{\b}_i, \, \bar{\b}_j) =0, \quad  (\bar{\a}_i, \,\bar{\b}_j)=\delta_{ij} q(\bar{\s}).
\end{equation}
If $q(\bar{\s})=0$ then $\bar{\a}_1, \, \bar{\b}_1, \ldots, \bar{\a}_g, \, \bar{\b}_g$ generate a totally isotropic subspace of the symplectic space $V$, whose dimension is at most $(2g)/2=g$. Since $g \geq 2$, even adding $\bar{\s}$  we cannot generate $V$, so this case cannot occur. It follows that $q(\bar{\s})=1$, namely $\s^2=z$. Thus, $[\a_i, \, \b_i]=z$ for all $i$, and since $[\a_i, \b_i]=[\a_i, \b_i^{-1}]$, equation \eqref{TR} gives $z^g=z$, hence $g$ is odd.

Conversely, assume that $g$ is odd and that $G$ is an extra-special group of order $2^{2g+1}$. Using the notation of the presentations \eqref{eq:H5} and \eqref{eq:G5}, set  
\begin{equation}
\mathscr{V} = 
\begin{cases}
(x_1, \,y_1,  \ldots, x_g, \, y_g, \, x_1y_1) & \text{if } \, G=\mathsf{H}_{2g+1}(\mathbb{F}_2) \\
(x_1, \,y_1,  \ldots, x_g, \, y_g, \, x_g)  & \text{if } \, G=\mathsf{G}_{2g+1}(\mathbb{F}_2).
\end{cases}
\end{equation}
 It is  straightforward to check that $\mathscr{V}$ is a generating vector, hence $G$ is an admissible quotient of  $\B_2(\Sigma_g)$.
\end{proof}

For an extra-special group $G$ of order $2^{2g+1}$, let us define $A(G)$ as the number of vectors in $V$ which are anisotropic with respect to the quadratic form $q$, namely
\begin{equation} \label{eq:A(G)}
A(G)= \# \, \{\bar{v} \in V \; \; | \; \; q(\bar{v})=1 \}.
\end{equation}

\begin{proposition} \label{prop:A(G)}
We have
\begin{equation} \label{eq:A(G)_value}
A(G)=
  \begin{cases}
  2^{2g-1} -  2^{g-1}   & \textrm{if }
    G=\mathsf{H}_{2g+1}(\mathbb{F}_2) \\
   2^{2g-1} +  2^{g-1}&
    \textrm{if }
    G=\mathsf{G}_{2g+1}(\mathbb{F}_2).
  \end{cases}
\end{equation}
\end{proposition}
\begin{proof}
Let us denote by $I(G)$ the number of isotropic vectors of  $(V, \, q)$, namely, the number of vectors $\bar{v} \in V$ such that $q(\bar{v})=0$. Then $I(G)+A(G)=|V|=2^{2g}$. Moreover, by definition of the Arf invariant $\operatorname{Arf}(q)$ given in footnote \ref{footnote:Arf}, we also have $I(G)-A(G)=(-1)^{\operatorname{Arf}(q)}2^g$. Subtracting the second relation from the first one, the claim follows.  
\end{proof}

We can now prove the analogue of Theorem \ref{thm:extraspecial} in the case $p=2$.

\begin{theorem} \label{thm:2-extraspecial}
Let $g\geq 3$ be an odd positive integer and $G$ an extra-special group of order $2^{2g+1}$. The number $N_g(G)$ of admissible epimorphisms $\varphi \colon \B_2(\Sigma_g) \to G$ is given by
\begin{equation} \label{eq:number_N_2-extraspecial}
N_g(G) =  2^{2g+1} A(G) \cdot |\mathsf{Sp}(2g, \, \mathbb{F}_2)|.
\end{equation}
\end{theorem}
\begin{proof}
Let $\varphi \colon \mathsf{B}_2(\Sigma_g) \to G$ be an admissible epimorphism, with generating vector $(\a_1, \, \b_1, \, \ldots, \a_g, \, \b_g, \,\s)$. The proof of Proposition \ref{prop:extra-special_2_congruence} shows that, given the projected data  
\begin{equation}
\bar{\a}_1, \, \bar{\b}_1, \, \ldots, \bar{\a}_g, \, \bar{\b}_g, \, \bar{\s}
\end{equation}
in $V$, we have $q(\bar{\s})=1$ and moreover $\bar{\a}_1, \, \bar{\b}_1, \ldots, \bar{\a}_g, \, \bar{\b}_g$ is an ordered symplectic basis of $V$.  So the number of possible projected data is $A(G) \cdot |\mathsf{Sp}(2g, \, \mathbb{F}_2)|$. 

Conversely, if $(\a_1, \, \b_1, \, \ldots, \a_g, \, \b_g, \,\s)$ is a vector of elements of $G$ such that 
$\bar{\a}_1, \, \bar{\b}_1, \, \ldots, \bar{\a}_g, \, \bar{\b}_g$ is an ordered symplectic basis for $V$ and $q(\bar{\s})=1$, it is straightforward to check that $(\a_1, \, \b_1, \, \ldots, \a_g, \, \b_g, \,\s)$ satisfies \eqref{B2}, \eqref{B3} and \eqref{B4} since  $G$  has nilpotency class $2$ and $\s^2=z\in Z(G)$, 
and satisfies \eqref{TR} since $g$ is odd.
Hence, we must understand how the projected data lift to $G$. The same argument used in the case $p$ odd shows that  the symplectic basis lifts in $2^{2g}$ different ways. On the other hand, $\bar{\s}$ lifts to two different elements, namely $\s$ and $\s z$. Summing up, $N_g(G)$ is obtained multiplying the number of projected data by  $2^{2g+1}$, so we get \eqref{eq:number_N_2-extraspecial}.  
\end{proof}

Using Theorem \ref{thm:2-extraspecial}, together with Remark \ref{rmk:action_auto} and Corollary \ref{cor:Aut-2-extraspecial}, we deduce

\begin{corollary} \label{cor:2-extraspecial_auto}
Let $g\geq 3$ be an odd positive integer and $G$ an extra-special group of order $2^{2g+1}$. The number $N_g^{\circ}(G)$ of admissible epimorphisms $\varphi \colon \B_2(\Sigma_g) \to G$, up to automorphisms of $G$, is given by
\begin{equation} \label{eq:number_N_2-extraspecial_auto}
N_g^{\circ}(G) = \frac{1}{|\operatorname{Aut}(G)|} N_g(G)= 2^{2g-1}(2^{2g} -1).
\end{equation}
In particular, $N_g^{\circ}(G)$ is independent of $G$.
\end{corollary}

\begin{remark} \label{rmk:difference-between-p-odd_2}
The case $p=2$ exhibits several features that are absent when $p$ is odd.

For an extraspecial $p$-group $G$, with $p$ odd, the quotient
$V=G/Z(G) \simeq \mathbb{F}_p^{2g}$ carries a non-degenerate alternating form induced by the commutator map. The admissibility condition is controlled by this symplectic structure: the image of the braid generator $\sigma \in \B_2(\Sigma_g)$  is a nontrivial central element $\s \in G$ and, after a suitable normalization of  $\mathsf{s}$, admissible generating vectors are obtained by lifting symplectic bases of  $V$.

The situation is markedly different when $p=2$. In this case, the symplectic space $V$ admits a canonical quadratic form $q \colon V \longrightarrow \mathbb{F}_2$, defined by $q(\bar v)=0$ if and only if $v^2=1$.
The admissibility condition is no longer determined by central elements, but rather by elements $\s\in G$ satisfying $\s^2=z$. Equivalently, the image $\bar{\s}\in V$ must be anisotropic with respect to $q$, i.e., we must have $q(\bar{\s})=1$. So, the image $\s \in G$ of the braid generator $\sigma \in \B_2(\Sigma_g)$ has order $4$, whereas for odd $p$ it has order $p$. This means, in particular, that if $G$ is an admissible extra-special quotient of  order $p^{2g+1}$, then the element $\varphi(\sigma)$ is never an involution. This is a striking contrast to the case of cyclic and dihedral quotients.

Thus, while the enumeration of admissible generating vectors for odd $p$  is governed by symplectic geometry, the case of characteristic $2$ also involves the finer structure of a quadratic space and its Arf invariant. This dichotomy is reflected in our counting formulas. For odd $p$, the quantity $N_g(G)$ is independent of $G$, whereas $N_g^\circ(G)$ depends on it, and so distinguishes the Heisenberg and non-Heisenberg types. For $p=2$, the opposite phenomenon occurs, namely $N_g(G)$ depends on the Arf invariant of $q$, and so on $G$, while $N_g^\circ(G)$ is independent of $G$.
\end{remark}

We end this section by investigating  the Chern slope of  surfaces arising from extra-special admissible quotients of  $\B_2(\Sigma_g)$.
\begin{proposition} \label{prop:slope_S_extra-special}
Let $G$ be an extra-special group of  order $p^{2g+1}$ and let $f \colon S \to \Sym^2 \Sigma_g$ be the $G$-cover associated with an admissible epimorphism $\varphi \colon \B_2(\Sigma_g) \to G$.   Let $\nu(S) = \nu(g, \, n)$ be the Chern slope of $S$, where $n=4$ if $p=2$ and $n=p$ if $p$ is odd. Then
\begin{equation}
2  < \nu(S) \leq  \frac{47}{21} \simeq 2.238095.
\end{equation}
 Moreover:
\begin{itemize}
\item $2$ is the infimum, but it is not attained$;$
 \item the upper bound $47/21$ occurs precisely when $(g, \, p)=(2, \, 3)$.
\end{itemize} 
 \end{proposition}
\begin{proof}
We distinguish the two cases $p=2$ and $p$ odd.

When $p=2$, putting  $n=4$ in \eqref{eq:nu(g,n)}, we get
\begin{equation} \label{eq:nu(g,n)_2_extraspecial}
\nu(S)=\nu(g, \, 4) = 2 + \frac{3}{8g-6}, \quad g \geq 3 \; \; \text{odd}.
\end{equation}  
This function has the  lower bound $2$ (which is not a minimum), whereas its maximum (attained at  $g=3$) equals 
$13/6 = 2.1 \bar{6}$. 

When $p$ is odd, we know that $g \equiv -1 \pmod{p}$. Setting $n=p$ and  $g=kp-1$ in \eqref{eq:nu(g,n)}, we get
\begin{equation} \label{eq:nu(g,n)_p_extraspecial}
\nu(S)=\nu(kp-1, \, p)  = 2 + \frac{p^2-4}{(2(kp-1)-1)p^2-2p}   = 2 + \frac{p^2-4}{p(2kp^2-3p-2)}.  
\end{equation}

We now observe that, for fixed $p$, the fraction in \eqref{eq:nu(g,n)_p_extraspecial} is strictly decreasing in $k$. So we get
\begin{equation}
\nu(S) \leq \nu(p-1, \, p) = 2 + \frac{p+2}{2p^2+p}.  
\end{equation}
For $x >0$, the real function $u(x)=(x+2)/(2x^2+x)$ is strictly decreasing. Consequently, its global maximum over odd primes $p$ is attained at  $p=3$. Consequently, when $k=1$ we get  $g=2$,  hence  the global maximum of $\nu(g, \, p)$ is $\nu(2, \, 3)=47/21$.
\end{proof}


\section{The first homology group of the Galois cover $S$} \label{sec:H_1}

Let us now explain how to compute the first homology group $H_1(S, \,\mathbb{Z})$, where $f \colon S \to \Sym^2 \Sigma_g$ is the $G$-cover associated with an admissible epimorphism  $\varphi \colon \B_2(\Sigma_g) \to G$ of type $(g, \, n)$. For this purpose, we will make use of the following result, which is a consequence of \cite[Theorem p. 254]{Fox57}.

\begin{proposition} \label{prop:fundamental-group-branched-cover}
  Let $G$ be a finite group and $\varphi \colon \mathsf{B}_2(\Sigma_g)
  \to G$ be an admissible epimorphism such that $\varphi(\sigma)$
  has order $n \geq 2$. If $f \colon S \to \Sym^2 \Sigma_g$
  is the $G$-cover associated with $\varphi$, then
  $\pi_1(S)$ fits into a short exact sequence
  \begin{equation} \label{eq:fundamental-group-branched-cover}
    1 \to \pi_1(S) \to	\mathsf{B}_2(\Sigma_g)^{\operatorname{orb}}
    \stackrel{\,\,\,\bar{\varphi}}{\to} G \to 1,
  \end{equation}
  where the \emph{orbifold braid group} $\mathsf{B}_2(\Sigma_g)^{\operatorname{orb}}$ is defined as the quotient of
  $\mathsf{B}_2(\Sigma_g)$ by the normal closure of the cyclic subgroup
  $\langle \sigma^n \rangle$, and $\bar{\varphi} \colon
  \mathsf{B}_2(\Sigma_g)^{\operatorname{orb}}  \to G$ is the group epimorphism
  naturally induced by $\varphi$.
\end{proposition}
Proposition \ref{prop:fundamental-group-branched-cover} allows one to compute $\pi_1(S)$, and so its abelianization $H_1(S,
\, \mathbb{Z})$. The calculations are usually unfeasible by hand, but can be carried out by  using the Computer Algebra System \verb|GAP4|, see   \cite{GAP4}. All our scripts  are publicly available in the GitLab repository \cite{GitLab26}.

Let us present the script in the case where $g=2$ and 
$G$ an extra-special group of order $3^5=243$, analyzing in detail the surfaces $S$ that we obtain.

We start by redefining the commutators according to our convention and implementing the function \verb|Multiplicities|, which will be useful later.

\begin{lstlisting}
##########################################################
### redefine commutators
###########################################################
comm:=function(x, y) return x*y*x^-1*y^-1; end;

####################################################################
# Multiplicities(list)
#
# Returns a list of pairs [element, multiplicity]
# for the distinct elements of the input list.
#################################################################### 

Multiplicities := function(list)
    local dict, keys, x;

    dict := NewDictionary([], true);
    keys := [];

    for x in list do
        if KnowsDictionary(dict, x) then
            AddDictionary(dict, x, LookupDictionary(dict, x) + 1);
        else
            AddDictionary(dict, x, 1);
            Add(keys, x);
        fi;
    od;

    return List(keys, x -> [x, LookupDictionary(dict, x)]);
end;
\end{lstlisting}

Then we construct the braid group $\B_2(\Sigma_2)$:

\begin{lstlisting}
################################################################
###  Construction of the braid group B_2(\Sigma_2)
################################################################
F:=FreeGroup("A1", "B1", "A2", "B2", "S");;
A1 := F.1;;
B1 := F.2;;
A2 := F.3;;
B2 := F.4;;
S  := F.5;;

B21:= comm(A1, S^-1*A1*S^-1);;
B22:= comm(B1, S^-1*B1*S^-1);;
B23:= comm(A2, S^-1*A2*S^-1);;
B24:= comm(B2, S^-1*B2*S^-1);;
 
B31:= comm(A2, S^-1*A1*S);;
B32:= comm(A2, S^-1*B1*S);;
B33:= comm(B2, S^-1*A1*S);;
B34:= comm(B2, S^-1*B1*S);;

B41:= comm(A1, S^-1*B1*S^-1)*S^-2;;
B42:= comm(A2, S^-1*B2*S^-1)*S^-2;;

BTR:=comm(A1, B1^-1)*comm(A2, B2^-1)*S^-2;; 

Braid:=F/[B21, B22, B23, B24, B31, B32, B33, B34, B41, B42, BTR];;
\end{lstlisting}

We now define the extra-special group $G=\mathsf{H}_5(\mathbb{F}_3)$ and we compute the generating vectors of type $(2, \, 3)$ on $G$. In order to speed up the computation, we will not compute the full list of generating vectors, but we make use of the $\operatorname{Aut}(G)$-action at some point:

\begin{lstlisting}
#########################################################
### Definition of G and computations of generating vectors
##########################################################

G:=ExtraspecialGroup(243, "+");;
AutG:=AutomorphismGroup(G);;

### s is central and nontrivial ###
GG:=[];;
for g in Center(G) do
if Order(g)>1 then
Add(GG, g);
fi; od;

### a1,a2,b1,b2 are not central ###
GGG:=Difference(G, Center(G));;

### (a1, b1, s) up to Aut(G)
list_1:=[];;

for s in GG do
for a1 in GGG do
for b1 in GGG do
R41:= comm(a1, b1)*s^-2;;

if Order(R41)=1 then
Add(list_1, [a1, b1, s]);
fi; od; od; od;

orb_list_1:=OrbitsDomain(AutG, list_1, OnTuples);;

### add a2, b2
list_2:=[];;

for l in orb_list_1 do
a1:=l[1][1];; b1:=l[1][2];; s:=l[1][3];;
 Ga1:=Centralizer(G, a1); 
 Gb1:=Centralizer(G, b1);
 Ga1b1:=Intersection(Ga1, Gb1);
 for a2 in Ga1b1 do
 for b2 in Ga1b1 do
 
R42:=comm(a2, b2)*s^-2;;
TR:=comm(a1, b1^-1)*comm(a2, b2^-1)*s^-2; 

H:=Subgroup(G, [a1, b1, a2, b2, s]);;

if 
Order(R42)=1 and
Order(TR)=1 and
Order(H)=Order(G) then
Add(list_2, [a1,b1,a2,b2,s]);
fi; od; od; od; 
\end{lstlisting}

By using the function \verb|Orbits| we can now compute the number of generating vectors of type $(2, \, 3)$ on $G$, up to automorphisms of $G$. 

Next, we construct the list \verb|representatives_generating_vectors| by choosing one representative from each $\operatorname{Aut}(G)$-orbit. The cardinality of such a list  is $1$, in agreement with Corollary \ref{cor:extraspecial_auto}.

\begin{lstlisting}
###############################################################
### Orbits of generating vectors up to Aut(G) 
###############################################################
orb_list_generating_vectors:=Orbits(AutG, list_2, OnTuples);;
Size(orb_list_generating_vectors);
1

##############################################################
### Creation of a list of representatives of generating vectors 
##############################################################
representatives_generating_vectors:=[];;
for l in orb_list_generating_vectors do
Add(representatives_generating_vectors, l[1]);
od;
Size(representatives_generating_vectors);
1
\end{lstlisting}

Finally, we apply Proposition \ref{prop:fundamental-group-branched-cover}: for each representative generating vector obtained above, we compute the first homology group $H_1(S, \, \mathbb Z)$ of the corresponding  $G$-cover.

\begin{lstlisting}
####################################################################
### Computing H_1(S,Z) for each representative ###
####################################################################
set_ab_inv:=[];;

for l in representatives_generating_vectors do
a1:=l[1];; b1:=l[2];; a2:=l[3];; b2:=l[4];; s:=l[5];; m:=Order(s);

A1 := Braid.1;;
B1 := Braid.2;;
A2 := Braid.3;;
B2 := Braid.4;;
S  := Braid.5;;

OrbBraid:=Braid/[S^m];;
A1 := OrbBraid.1;;
B1 := OrbBraid.2;;
A2 := OrbBraid.3;;
B2 := OrbBraid.4;;
S  := OrbBraid.5;;

hom:=GroupHomomorphismByImages(OrbBraid, G ,[A1, B1, A2, B2, S],
[a1, b1, a2, b2, s]);;
K:=Kernel(hom);; 
Add(set_ab_inv, AbelianInvariants(K)); 
od;

### Reduced list of abelian invariants with occurrences ###
M:=Multiplicities(set_ab_inv);;
for m in M do
Print(m[1], " "); Print(" "); Print (m[2], "\n");
od;
[ 0, 0, 0, 0, 2, 2, 2, 2, 2, 2, 2, 2, 2, 2, 2, 2, 2, 2, 2, 2, 2, 2 ]  1
\end{lstlisting} 

The output of our function \verb|Multiplicities| shows that, for the unique generating vector in our list of representatives, we have
\begin{equation}
H_1(S, \, \mathbb{Z})=\mathbb{Z}^4 \oplus \mathbb{Z}_2^{18}.
\end{equation} 

We can now apply the same script to make similar computations in the case $G= \mathsf{G}_5(\mathbb{F}_3)$. We only need to replace the line 
\begin{lstlisting}
G:=ExtraspecialGroup(243, "+");;
\end{lstlisting}
with the line
\begin{lstlisting}
G:=ExtraspecialGroup(243, "-");;
\end{lstlisting}
At the end of the computation, we get
\begin{lstlisting}
### Reduced list of abelian invariants with occurrences ###
M:=Multiplicities(set_ab_inv);;
for m in M do
Print(m[1], " "); Print(" "); Print (m[2], "\n");
od;
[ 0, 0, 0, 0] 80
\end{lstlisting}
This shows that  $G$ admits  $80$ equivalence classes of generating vectors of type $(2, \, 3)$, up to automorphisms of $G$ (again, in agreement with Corollary \ref{cor:extraspecial_auto}) and that, for all the surfaces $S$ arising as the  corresponding $G$-covers $f \colon S \to \operatorname{Sym}^2 \Sigma_2$, we have
\begin{equation}
H_1(S, \, \mathbb{Z})=\mathbb{Z}^4.
\end{equation}

Combining Proposition $\ref{prop:invariants_S}$ with the preceding homology computations, we obtain the following result.

\begin{proposition} \label{prop:invariants_H_G}
Let $G$ be an extra-special group of order $3^5$, and let $f \colon S \to \operatorname{Sym}^2 \Sigma_2$ be the $G$-cover corresponding to an admissible epimorphism $\varphi \colon \B_2(\Sigma_2) \to G$. 
Then $S$ is a minimal surface of general type with 
\begin{equation}
p_g(S)=154, \quad q(S)=2, \quad K_S^2=1269, \quad  \operatorname{sign}(S)=45
\end{equation}
and whose Betti numbers are
\begin{equation}
b_0(S)=1, \quad b_1(S)=4, \quad  b_2(S)=573, \quad b_3(S)=4, \quad b_4(S)=1 
\end{equation}
Moreover
\begin{itemize}
\item if $G=\mathsf{H}_5(\mathbb{F}_3)$, then
\begin{equation}
H_1(S, \, \mathbb{Z})=\mathbb{Z}^4 \oplus \mathbb{Z}_2^{18};
\end{equation}
\item if $G=\mathsf{G}_5(\mathbb{F}_3)$, then
\begin{equation}
H_1(S, \, \mathbb{Z})=\mathbb{Z}^4.
\end{equation}
\end{itemize}
\end{proposition}

\begin{remark}\label{rmk:invariants_H_G}
   Since our construction is topological, we can start from an \emph{arbitrary} Riemann surface $\Sigma_2$. So, we obtain two $3$-dimensional families of minimal surfaces of general type such that 
\begin{itemize}
\item any two surfaces lying in different families have the same numerical invariants
\begin{equation}
p_g=154,  \quad q=2, \quad K^2=1269,
\end{equation}
 the same Betti numbers and the same \emph{positive} signature;
\item any two surfaces lying in different families \emph{are not homotopically equivalent}, since they have different torsion part of their first homology group. 
\end{itemize}
\medskip

To justify the dimension claim, recall that, for a curve
$C$ of genus $2$, the Abel-Jacobi map
$\operatorname{Sym}^2 C \to J(C)$ is birational.
Hence the composition
\begin{equation}
S \longrightarrow \operatorname{Sym}^2 C
\longrightarrow J(C)
\end{equation}
is generically finite and surjective. Since $q(S)=2$,
the induced homomorphism
$\operatorname{Alb}(S) \to J(C)$ is an isogeny.
This means that, if $S$ is fixed, the Jacobian $J(C)$ belongs
to the isogeny class of $\operatorname{Alb}(S)$,
which contains only countably many isomorphism
classes of principally polarized abelian varieties.
By the Torelli theorem, only countably many curves
$C$ can therefore give rise to a fixed surface $S$.

Now, let $\mathcal H_G$ be the algebraic Hurwitz space parametrizing the $G$-covers
 $S \to \operatorname{Sym}^2 \Sigma_2$ under consideration. 
 It is of finite type over $\mathbb C$,
and the natural map $\mathcal H_G \to \mathcal M_2$
is surjective with finite fibres. Hence
$\dim \mathcal H_G=3$.
The construction induces an algebraic moduli map
\begin{equation}
\mu \colon \mathcal H_G \longrightarrow \mathcal M_S,
\end{equation}
where $\mathcal{M}_S$ denotes the moduli space of surfaces
of general type $S$ with the prescribed invariants.
By the preceding argument, every fibre of $\mu$ is
at most countable, since only countably many curves
can give rise to a fixed surface, and each curve
admits only finitely many such covers.
Since the fibres are algebraic spaces of finite type
over $\mathbb C$, they must be zero-dimensional,
hence finite. Thus $\mu$ is quasi-finite, and
its image has dimension $3$.
\end{remark}

\section*{Acknowledgements}
A preliminary version of some of the results presented here was contained in the PhD dissertation of Michelangelo Migliano; he thanks Matteo Penegini and Roberto Pignatelli for their careful reading of the manuscript and their useful comments.

Massimiliano Alessandro is a member of GNSAGA-INdAM.

Francesco Polizzi was partially supported by GNSAGA-INdAM. Part of this work was done in January 2026, when he visited the Universit\"{a}t des Saarlandes. He is grateful to the members of the Arbeitsgruppe Algebraic and Complex Geometry  for the invitation and the hospitality. He also thanks Giovanni Cutolo and Nicola Durante for insightful conversations on orthogonal geometry in characteristic $2$.

\end{document}